\documentclass[12pt]{article}

\usepackage{epsfig}

\newcommand{\me} {1/2}
\newcommand{\mme} {-1/2}
\newcommand{\tme} {3/2}

\newcommand{\nuA} {{\nu_A}}
\newcommand{\eps}{\varepsilon}

\newcommand{\vphi}{\varphi}

\newtheorem{theorem}{Theorem}[section]
\newtheorem{lemma}[theorem]{Lemma}
\newtheorem{corollary}[theorem]{Corollary}
\newtheorem{remark}[theorem]{Remark}

\def\N{{\mathord{\rm I\mkern-3.6mu N}}}

\def\R{{\mathord{\rm I\mkern-3.6mu R}}}
\def\Z{{\mathord{\rm Z\mkern-5.6mu Z}}}

\title{AN ITERATIVE METHOD FOR SOLVING ELLIPTIC CAUCHY PROBLEMS}

\author{A. Leit\~ao \\
Department of Mathematics \\ Federal University of Santa Catarina \\
P.O. Box 476, 88010-970 Florian\'opolis, Brazil }

\date{}

\begin{document}

\maketitle

\begin{abstract}

        We investigate the Cauchy problem for elliptic operators with 
$C^\infty$--coefficients at a regular set $\Omega \subset \R^2$, which is 
a classical example of an ill-posed problem. The Cauchy data are given at 
the subset $\Gamma \subset \partial\Omega$ and our objective is to 
reconstruct the trace of the $H^1(\Omega)$ solution of an elliptic equation 
at $\partial \Omega / \Gamma$. The method described here is a generalization 
of the algorithm developed by Maz'ya et al. [Ma] for the Laplace operator, 
who proposed a method based on solving successive well-posed mixed boundary 
value problems (BVP) using the given Cauchy data as part of the boundary data. 
We give an alternative convergence proof for the algorithm in the case we 
have a linear elliptic operator with $C^\infty$--coefficients. We also present 
some numerical experiments for a special non linear problem and the obtained 
results are very promisive.
\end{abstract}

\pagestyle{plain}

\section{Introduction}

%----------------------------------------------------------------------------%
\subsection{Main results}

        The algorithm of Maz'ia et al. [Ma] is formulated here for general 
elliptic operators. A new convergence proof for this iterative algorithm 
using a functional analytical approach is given in section~2.1 (see Theorem~%
\ref{Tk_konverg}), where we describe the iteration using powers of an affine 
operator $T$. The key of the proof is to define an alternative topology 
(see Lemma~\ref{neue_norm}) for the space $H^{\me}_{00}(\Gamma)'$ -- where 
the iteration is considered -- and to prove that the linear part of $T$ 
satisfies special properties (see Theorem~\ref{mazya}). The converse of 
Theorem~\ref{Tk_konverg} is also proved, i.e. if the iteration converges, 
it's limit is the solution of the Cauchy Problem.

        Some properties of $T_l$ (the linear part of $T$) such as 
positiveness, self adjointness and injectivity are verified in section~2.1 
(see Theorem~\ref{Tl_eigenschaft}). In section~2.2 we prove a spectral 
property of $T_l$, that is attached to the ill-posedness of the elliptic 
Cauchy Problem.

        In section~2.3 we analyze the convergence speed of the iteration 
for the special case when the spectral decomposition of $T_l$ is known. 
The effectiveness of two regularization schemas based on the spectral 
decomposition of $T_l$ (the linear part of $T$) is considered in section~2.4.

        In section~3 some numerical experiments are presented, where we test 
the algorithm performance for linear consistent, linear inconsistent and non 
linear Cauchy problems.

        An analysis of iterative method in the special case of a square region 
can be found in [JoNa]. The idea of this method is also applied to hyperbolic 
operators in [Bas] that uses semi-group theory in his approach.

%----------------------------------------------------------------------------%
\subsection{About Cauchy problems}

        Let $\Omega \subset \R^2$ be an open, bounded and simply connected 
set. As an {\em elliptic Cauchy problem} at $\Omega$ we consider an (time 
independent) initial value problem for an elliptic differential operator 
defined over $\Omega$, where the initial data is given at the manifold 
$\Gamma \subset \partial \Omega$. 

        The problem we analyze is to evaluate the trace of the solution of 
such an initial value problem at the part of the boundary where no data 
was prescribed, actually at $\partial\Omega \backslash \Gamma$. As a 
solution of our Cauchy problem we consider a $H^1(\Omega)$--distribution, 
which solves the weak formulation of the elliptic equation in $\Omega$ 
and also satisfies the Cauchy data at $\Gamma$ in the sense of the trace 
operator.

        It's well known that elliptic Cauchy problems are ill--posed. 
According to the definition of Hadamard an initial value problem (IVP) or a 
BVP is said to be well--posed, when the following three conditions are 
satisfied:%
\footnote{More details in [Bau] or [Lo].}
existence and unicity of solutions, and continuous dependence of the data. 
The next example was encountered by Hadamard himself [Had] and shows that 
the solution of an elliptic Cauchy problem may not depend continuously of 
the initial data. One analyzes the family of problems:
$$   \left\{  \begin{array}{rl}
         \Delta u_k \ = \ 0\, , & (x,y)\in \Omega=(0,1)\times(0,1) \\
         u_k(x,0)   \ = \ 0\, , & x \in (0,1) \\
         \frac{\partial}{\partial y} u_k(x,0) \ = \ \vphi_k\, , & x\in (0,1)
     \end{array} \right.  $$
where $\vphi_k(x) = (\pi k)^{-1} sin(\pi k x)$. The respective solutions
$$           u_k(x,y)\ =\ (\pi k)^{-2} sinh(\pi k y)\ sin(\pi k x)          $$
do exist for every $k\in\N$ and they are unique. The sequence $\{\vphi_k\}$ 
converges uniformly to zero. Taking the limit $k \to \infty$ we have a Cauchy 
problem with homogeneous data, which admits only the trivial solution. But 
for every fixed $y>0$ the solutions $u_k$ oscillate stronger and stronger and 
become unbounded as $k \to \infty$. Consequently the sequence $u_k$ does not 
converge to zero in any reasonable topology.

        If in this example one takes for Cauchy data the $C^\infty$--functions 
$(f,g)$ instead of $(0,\vphi_k)$, it is possible to show (see [GiTr]) that if 
$f \equiv 0$, then $g$ must be analytical. This means that a {\em classical 
solution} may not exist, even if one uses smooth functions as Cauchy data.

        The unique well-posedness condition that is satisfied for this 
problem is the second one. With adequate arguments it is possible to extend 
the Cauchy--Kowalewsky and Holmgren Theorem to the $H^1$--context in 
order to guarantee uniqueness of solutions also in weak sense (see 
Theorem~\ref{schwach_eindeut}).

%----------------------------------------------------------------------------%
\subsection{Description of the algorithm}

        Let $\Omega$ be an open set in $\R^2$ with smooth boundary 
$\partial\Omega$, which is divided in two open and connected components: 
$\Gamma_1$ and $\Gamma_2$, such that $\Gamma_1\cap\Gamma_2 = \emptyset$ and 
$\overline{\Gamma_1\cup\Gamma_2} = \partial\Omega$. Let $P$ be the second 
order elliptic differential operator defined by:
\begin{equation} \label{L-definition}
     P(u) \ := \ - \sum_{i,j=1}^{2}\,{D_i (a_{i,j} D_j u)}
\end{equation}
\noindent  where the real functions $a_{i,j}$ satisfy
\begin{equation}
\left\{ \begin{array}{l}
        -\ a_{i,j} \in L^\infty(\Omega); \\
        -\ \mbox{the matrix } A(x) := (a_{i,j})_{i,j=1}^2
        \mbox{ satisfies: } \xi^t A(x)\, \xi > \alpha ||\xi||^2, \\
        \mbox{ \ \ \ a.e. } x \in \Omega,\ \forall \xi \in\R^2 \mbox{ where }
        \alpha>0 \mbox{ is given (independent of } x).
\end{array}\right.
\label{coef_beding}
\end{equation}
        Given the Cauchy data $(f,g) \in H^{\me}(\Gamma_1) \times 
H^{\me}_{00}(\Gamma_1)'$, we search for a $H^1$--solution of the problem%
\footnote{Details about the notation can be found in Appendix~A.}

\medskip \noindent
$   (CP)     \hfill \left\{  \begin{array}{rl}
                          P u = 0 & ,\ \mbox{in}\ \Omega    \\
                          u = f          & ,\ \mbox{at}\ \Gamma_1 \\
                          u_\nuA = g    & ,\ \mbox{at}\ \Gamma_1
             \end{array} \right. .  \hfill              $
\medskip

        Our objective is to reconstruct the trace of the solution $u$ and 
it's conormal derivative at $\Gamma_2$.%
\footnote{Note that if we knew the conormal derivative at $\Gamma_2$, $u$ 
could be evaluated as the solution of a mixed BVP.}
Given the approximation $\varphi_0 \in H^{\me}_{00}(\Gamma_2)'$ for 
$u_{\nuA|_{\Gamma_2}}$, we define the sequence $\{ \vphi_k \}_{k\in\N}$ 
using the following iteration rule:

\medskip \noindent
$  (IT)  \hfill \left\{ \begin{array}{lccc}
                w \in H^1(\Omega)\ \mbox{solve:} & P w = 0; &
                w_{|_{\Gamma_1}} = f; & w_{\nuA|_{\Gamma_2}} = \vphi_k; \\
                \psi_k := w_{|_{\Gamma_2}}; & & & \\
                 & & & \\
                v \in H^1(\Omega)\ \mbox{solve:} & P v = 0; &
                v_{\nuA|_{\Gamma_1}} = g; & v_{|_{\Gamma_2}} = \psi_k; \\
                \vphi_{k+1} := v_{\nuA|_{\Gamma_2}}. & & & 
         \end{array} \right.   \hfill $
\medskip

\noindent  In (IT) two differential equations are solved and two trace 
operators are applied. Actually we generate two sequences: the first one of 
Dirichlet traces and the second one of Neumann traces, both defined at 
$\Gamma_2$. As the functions $w$ and $v$ are both in $H^1(\Omega;P)$, one 
concludes from Theorems~\ref{spursatz_rand_st} and \ref{spur_H_om_P_rand_st} 
respectively that $\{ \vphi_k \} \subset H^{\me}_{00}(\Gamma_2)'$ and 
$\{ \psi_k \} \subset H^{\me}(\Gamma_2)$.

\begin{remark}  If the Neumann data $g$ of (CP) is a $H^{\mme}(\Gamma_2)$--%
distribution, one proves using the Theorems of Appendix~C that the 
sequence $\{ \vphi_k \}$ can be defined on the Sobolev space 
$H^{\mme}(\Gamma_2)$.
\end{remark}

\begin{remark}  If one supposes $\partial\Omega = \Gamma_1 \cup \Gamma_2 \cup 
\Gamma_3$ and wants to analyze a Cauchy problem with data given at $\Gamma_1$ 
plus a further boundary condition (Neumann, Dirichlet, $\dots$) at $\Gamma_3$, 
it is possible to adapt the iteration by adding this boundary condition at 
$\Gamma_3$ to both BVP in (IT). This over--determination of boundary data 
does not affect the analysis of the algorithm.
\end{remark}

%----------------------------------------------------------------------------%
\subsection{Functional--analytical approach}

        The main objective in this section is to represent the iteration (IT) 
using an operator $T: H^{\me}_{00}(\Gamma_2)' \to H^{\me}_{00}(\Gamma_2)'$. 
We define the operators $L_n : H^{\me}_{00}(\Gamma_2)' \to H^1(\Omega)$ and 
$L_d : H^{\me}(\Gamma_2) \to H^1(\Omega)$ by:
$$    L_n (\varphi) := w \in H^1(\Omega)\ \ \ \ \ \ and\ \ \ \ \ \
                                        L_d (\psi) := v \in H^1(\Omega),    $$
\noindent  where the functions $w$ and $v$ are respectively solutions of the 
BVP's 

\medskip \noindent
$  \hfill \left. \begin{array}{ccc}
             P w = 0\ \mbox{\ in}\ \Omega;   & w_{|_{\Gamma_1}} = f; &
             w_{\nuA|_{\Gamma_2}} = \varphi \\
             {\hskip-8.3cm \mbox{and}} & & \\
             P v = 0\ \mbox{\ in}\ \Omega;   & v_{\nuA|_{\Gamma_1}}=g; &
             v_{|_{\Gamma_2}} = \psi
   \end{array} \right. \hfill       $
\medskip

\noindent  With the aid of the Neumann trace operator $\gamma_n: 
H^1(\Omega, P) \rightarrow H^{\me}_{00} (\Gamma_2)'$, $\gamma_n(u) := 
u_{\nuA|_{\Gamma_2}}$ and the Dirichlet trace operator $\gamma_d : H^1(\Omega) 
\rightarrow H^{\me}(\Gamma_2)$, $\gamma_d(u) := u_{|_{\Gamma_2}}$ one can 
rewrite (IT) as
\begin{equation} 
        \left\{ \begin{array}{cc}
        w \ = \ L_n (\varphi_k);   & \psi_k \ = \ \gamma_d(w) \\
        v \ = \ L_d (\psi_k);   & \varphi_{k+1} \ = \ \gamma_n(v)
        \end{array} \right. 
\label{T_schritt} \end{equation}
\noindent   If we define $T := \gamma_n \circ L_d \circ \gamma_d 
\circ L_n$, we conclude immediately that $T$ is an affine operator on 
$H^{\me}_{00}(\Gamma_r)'$, which satisfies
$$            \vphi_{k+1} \ = \ T(\vphi_k) \ = \ T^{k+1}(\vphi_0).          $$
\noindent  That means we are able to describe the iteration (IT) with 
powers of the operator $T$. As $L_n$ and $L_d$ are both affine, we can write
$$  L_n(\cdot) \ = \ L_n^l(\cdot) \ + \ w_f \hskip1cm \mbox{and} \hskip1cm
    L_d(\cdot) \ = \ L_d^l(\cdot) \ + \ v_g,
$$
\noindent  where the $H^1(\Omega, P)$--functions $w_f$ and $v_g$ depend 
only of $f$ and $g$ respectively. With these definitions we have
\begin{eqnarray}
\vphi_{k+1} \ = \ T(\vphi_k) & = &
      \underbrace{\gamma_n \circ L_d^l \circ \gamma_d \circ
        L_n^l(\vphi_k)}_{T_l(\vphi_k)} \ + \
      \underbrace{\gamma_n \circ L_d^l  \circ \gamma_d(w_f) + 
        \gamma_n(v_g)}_{z_{f,g}}                           \label{Tl_def} \\
   &   & \nonumber \\
   & = & T_l^{k+1}(\vphi_0) \ + \ \sum_{j=0}^k{T_l^j(z_{f,g})}.    \nonumber
\end{eqnarray}
\begin{remark}  If we set $\overline{\vphi} = \gamma_n u$, where $u$ is the 
solution of (CP), it follows from (IT) that $T\, \overline{\vphi} = 
\overline{\vphi}$. Conversely, if $\overline{\vphi}$ is a fixed point of the 
operator $T$, the functions $w$ and $v$ in (IT) have the same traces at 
$\Gamma_2$. From the uniqueness Theorem~\ref{schwach_eindeut} follows 
$w = v$ and they are both solutions of (CP).
\label{fix_punkt_bemerk}  \end{remark}

\section{Analysis of the method}

%----------------------------------------------------------------------------%
\subsection{Convergence proof}

        In order to study the iterative method proposed in section 1.3, we 
begin with equipping the space $H^{\me}_{00}(\Gamma_2)'$ with a new topology.

%=====================================%
\begin{lemma}  Let the coefficients of $P$ satisfy the conditions in 
(\ref{coef_beding}). The functional 
$$  ||\vphi||_* := \left( \int_\Omega{ (\nabla L_n^l (\vphi))^t  A \, 
                         (\nabla L_n^l (\vphi) ) \, dx} \right)^{\me}   $$
defines on $H^{\me}_{00}(\Gamma_2)'$ a norm, that is equivalent to the usual 
Sobolev norm of this space.
\label{neue_norm} \end{lemma}

\begin{proof}
        Given $\vphi \in H^{\me}_{00}(\Gamma_2)'$, the function $u = 
L_n^l(\vphi)$ solves following BVP:
$$  \left\{ \begin{array}{ccl}
            P \, u     & = & 0\ , \ \mbox{ in } \Omega \\
            u          & = & 0\ , \ \mbox{ at } \Gamma_1 \\
            u_{\nu_A}  & = & \vphi\ ,\ \mbox{ at } \Gamma_2
            \end{array} \right.                                             $$
\noindent  From Theorem~\ref{ex_eind_gemischt} one concludes $L_n^l(\vphi)$ is 
the unique solution in $H^1_0(\Omega\cup\Gamma_2)$. In the same theorem the 
continuous dependence of the data is proved, and from this follows
$$  ||\vphi||_*\ \leq\ c_1 \ ||L_n^l(\vphi)||_{H^1(\Omega)}\ \leq\
                            c_2 \ ||\vphi||_{H^{\me}_{00}(\Gamma_2)'} \, ,  $$
\noindent  where the first inequality follows from the norm equivalence 
between $||\cdot||_{H^1 (\Omega)}$ \ and \ $\langle A\nabla\cdot, 
\nabla\cdot \rangle_{L^2(\Omega)}^{\me}$ at $H^1_0(\Omega\cup\Gamma_2)$.

        The opposite inequality follows from the continuity of the Neumann 
trace operator in Theorem~\ref{spur_H_om_P_rand_st} and the norm equivalence 
used just above.
\qquad\end{proof}

\begin{remark}  Actually one can prove that the norm $||\vphi||_*$ is 
defined by an inner product and the space $H^{\me}_{00}(\Gamma_2)'$ is a 
Hilbert space with the inner product
$$  \langle \vphi, \psi \rangle_* \ = \ \int_{\Omega}
        { ( \nabla L_n^l(\vphi) )^t  A \, ( \nabla L_n^l(\psi) )\ dx} \, .  $$
\end{remark}

        In the next theorem we investigate some properties of the operator 
$T_l$, defined in section 1.4, when we equip the space $H^{\me}_{00} 
(\Gamma_2)'$ with the Hilbert space structure defined by $\langle \cdot, 
\cdot \rangle_*$.

%=====================================%
\begin{theorem}  Let $T_l \in {\cal L}(H^{\me}_{00}(\Gamma_2)')$ be the 
operator defined in (\ref{Tl_def}). The following assertions hold:
\begin{enumerate} \renewcommand{\labelenumi}{\roman{enumi})}
  \item  $T_l$ is positive;
  \item  1 is not an eigenvalue of $T_l$;
  \item  $T_l$ is self adjoint;
  \item  $T_l$ is injective.
\end{enumerate}
\label{Tl_eigenschaft} \end{theorem}

\begin{proof}
        {\it i)} We define the operator $W: H^{\me}_{00}(\Gamma_2)' 
\rightarrow H^1(\Omega)$ \ by \ $W(\vphi) := L_d^l \circ \gamma_d \circ 
L_n^l(\vphi)$, where the operators $L_d^l$, $L_n^l$ and $\gamma_d$ are the 
same as in section 1.4. From Theorems~\ref{green_allgemein} and 
\ref{H_me00_eigenschaft} follows for $\vphi, \psi \in H^{\me}_{00}(\Gamma_2)'$
\begin{eqnarray}
\int_{\Omega} \left( \nabla L_n^l T_l(\vphi) - \nabla W(\vphi) \right)^t 
\!\!\!\!& A &\!\!\!\! ( \nabla L_n^l(\psi) ) \, dx \ = \label{Jour1} \\
& = & \int_{\Omega}{P \left( L_n^l T_l(\vphi) - W(\vphi) \right)
                       L_n^l(\psi) \, dx} \nonumber \\
& & \!\!\!\! + \int_{\Gamma_1\cup\Gamma_2}{\left( L_n^l T_l(\vphi) - W(\vphi)
         \right)_{\nu_A}  L_n^l(\psi) d\Gamma} \ = \ 0 . \nonumber
\end{eqnarray}
From an analogous argument we have
\begin{equation}
\mbox{}\ \ \ \ \ 
\int_{\Omega}{\left( \nabla W(\vphi) - \nabla L_n^l(\vphi) \right)^t 
              A (\nabla W(\psi) ) \, dx} \ = \ 0 .
\label{Jour2} \end{equation}
         If we denote by $\langle \cdot,\cdot \rangle$ the inner product on 
$L^2(\Omega)$, it follows from (\ref{Jour1}) and (\ref{Jour2})
\begin{eqnarray*}
\langle T_l \, \vphi, \vphi \rangle_* & = & 
         \langle A \, \nabla L_n^l T_l(\vphi), \nabla L_n^l(\vphi) \rangle \\
& \stackrel{(\ref{Jour1})}{=} &
                 \langle A \, \nabla W(\vphi), \nabla L_n^l(\vphi) \rangle \\
& \stackrel{(\ref{Jour2})}{=} &
                     \langle A \, \nabla W(\vphi), \nabla W(\vphi) \rangle \\
& \geq & c\, ||W(\vphi)||^2_{H^1(\Omega)} ,
\end{eqnarray*}
for every $\vphi \in H^{\me}_{00}(\Gamma_2)'$.

        {\it ii)} Let us suppose there exists a $\vphi \in H^{\me}_{00} 
(\Gamma_2)'$, such that $T_l \, \vphi = \vphi$. Define $w := L_n^l(\vphi)$ 
and $v := L_d \circ \gamma_d \circ L_n^l(\varphi)$. For the difference 
$v - w$ we have:
$$     (v - w)_{|_{\Gamma_2}} \ = \ (v - w)_{\nu_A|_{\Gamma_2}} \ = \ 0.    $$
\noindent  From the unicity Theorem~\ref{schwach_eindeut} we have $v = w$. 
The definition of $w$ and $v$ imply $0 = w_{|_{\Gamma_l}}$ and 
$0 = v_{\nu_A|_{\Gamma_1}} = w_{\nu_A|_{\Gamma_1}}$. 
Theorem~\ref{schwach_eindeut} now implies $\vphi= 0$.

{\it iii)} analogous to (\ref{Jour1}) and (\ref{Jour2}) one proves that for 
$\vphi, \psi \in H^{\me}_{00}(\Gamma_2)'$ the identities
\begin{equation}
\langle A \, \nabla L_n^l(\vphi), \nabla L_n^l T_l(\psi) \rangle \ = \
\langle A \, \nabla L_n^l(\vphi), \nabla W(\psi) \rangle \label{Jour3}
\end{equation}
and
\begin{equation}
\langle A \, \nabla W(\vphi), \nabla W(\psi) \rangle \ = \
\langle A \, \nabla W(\vphi), \nabla L_n^l(\psi) \rangle
\end{equation}
hold. These last equations imply
$$ \langle T_l\,\vphi,\psi \rangle_* \ = \ \langle \vphi,T_l\,\psi \rangle_* ,
                      \ \forall\ \vphi, \psi \in H^{\me}_{00}(\Gamma_2)' .  $$

        {\it iv)} Take $\vphi_1$, $\vphi_2$ in $H^{\me}_{00}(\Gamma_2)'$ with 
$T_l \, \vphi_1 = T_l \, \vphi_2$. Define now $w := L_n^l (\vphi_1 - \vphi_2)$ 
and $v := L_d^l \circ \gamma_d \circ L_n^l(\vphi_1 - \vphi_2)$. We clearly 
have $v_{\nu_A|_{\Gamma_1}} = 0$ and the hypothesis $T_l(\vphi_1 - \vphi_2) 
= 0$ implies $v_{\nu_A|_{\Gamma_2}}  =  0$. Since $v$ satisfies $P \, v = 0$, 
we conclude that $v$ is constant in $\Omega$. From $v_{|_{\Gamma_2}} = 
w_{|_{\Gamma_2}} \in H^{\me}_{00}(\Gamma_2)$ follows $v \equiv 0$.%
\footnote{The unique constant function in $H^{\me}_{00}(\Gamma_l)$ is the 
null function.}
Then we have $w \equiv 0$ in $\Omega$ and the equality $\vphi_1 = \vphi_2$ 
follows.
\qquad\end{proof}

        In the next theorem we verify two properties of $T_l$, that are needed 
in the convergence proof of the iterative method described in section 1.3.

%=====================================%
\begin{theorem}  Let $T_l \in {\cal L}(H^{\me}_{00}(\Gamma_2)')$ be the 
operator defined in (\ref{Tl_def}). The following assertions are valid:
\begin{enumerate} \renewcommand{\labelenumi}{\roman{enumi})}
        \item  $T_l$ is regular asymptotic in $H^{\me}_{00}(\Gamma_2)'$, i.e. 
$\lim_{k \rightarrow \infty}{ ||T_l^{k+1}(\vphi) - T_l^k(\vphi)||_*} = 0,\ 
\forall \vphi \in H^{\me}_{00}(\Gamma_2)'$;
        \item  The operator $T_l$ is non expansive, i.e. 
$||T_l||_{ {\cal L}(H^{\me}_{00}(\Gamma_2)') } \leq 1$.
\end{enumerate}
\label{mazya} \end{theorem}

\begin{proof}
        {\it i)} Because of the identity $(T_l^{k+1}(\vphi) - T_l^k(\vphi)) 
= T_l^k(T_l - I)(\vphi)$, it is enough to prove that $T_l^k(\vphi_0) 
\rightarrow 0$ for every $\vphi_0 \in$ Rg$(T_l - I)$. Take $\psi \in 
H^{\me}_{00}(\Gamma_2)'$ with $(T_l - I)\psi = \vphi$. Note that we can 
describe the iteration $T_l^k(\psi)$ using the functions
\begin{equation} 
        \left\{ \begin{array}{lll}
        w_k(\psi) & = \ L_n^l(\gamma_n(v_{k-1}(\psi))),& k\geq 1  \\
        v_k(\psi) & = \ L_d^l(\gamma_d(w_k(\psi))),    & k\geq 0
        \end{array} \right.
\label{Tl_schritt} \end{equation}
and $w_0(\psi) = L_n^l(\psi)$. From (\ref{Tl_schritt}) follows 
$w_k = (v_k)_\nuA = 0$ at $\Gamma_1$, $(w_k)_\nuA = (v_{k-1})_\nuA$ and 
$v_k = w_k$ at $\Gamma_2$. These identities and Theorem~\ref{green_allgemein} 
give us
\begin{eqnarray*}
    \int_\Omega{ (\nabla v_k)^t  A \, (\nabla v_k) \, dx} & = &
    \int_{\Gamma_2}{ v_k (v_k)_\nuA \, d\Gamma } , \\
    \int_\Omega{ (\nabla w_k)^t  A \, (\nabla w_k) \, dx} & = &
    \int_{\Gamma_2}{ w_k (w_k)_\nuA \, d\Gamma } , \\
    \int_\Omega{ (\nabla v_k)^t  A \, (\nabla w_k) \, dx} & = &
    \int_{\Gamma_2}{ w_k (v_k)_\nuA \, d\Gamma } , \\
    \int_\Omega{ (\nabla v_{k-1})^t  A \, (\nabla w_k) \, dx} & = &
    \int_{\Gamma_2}{ w_k (v_{k-1})_\nuA \, d\Gamma } .
\end{eqnarray*}
From these identities we obtain
\begin{eqnarray}
& & \int_\Omega{ \nabla (w_k-v_{k-1})^t  A \, \nabla (w_k-v_{k-1}) \, dx} \ =
\label{GlMa1} \\
& & = \ \int_\Omega{ \left[ (\nabla v_{k-1})^t  A \, (\nabla v_{k-1}) \, - \,
                  (\nabla w_k)^t  A \, (\nabla w_k) \right] \, dx} \nonumber
\end{eqnarray}
and
\begin{eqnarray}
& & \int_\Omega{ \nabla (v_k - w_k)^t  A \, \nabla (v_k - w_k) \, dx} \ =
\label{GlMa2} \\
& & = \ \int_\Omega{ \left[ (\nabla w_k)^t  A \, (\nabla w_k) \, - \, 
                  (\nabla v_k)^t  A \, (\nabla v_k) \right] \, dx} . \nonumber
\end{eqnarray}
Note that the definition of $\vphi$ and $\psi$ imply $w_k(\vphi) = w_{k+1} 
(\psi) - w_k(\psi)$. Equations (\ref{GlMa1}) and (\ref{GlMa2}) now imply

\medskip \noindent $\displaystyle
\int_\Omega{ (\nabla w_k(\vphi))^t  A \, (\nabla w_k(\vphi)) \, dx} \ \le $
\hfill

\hfill $\displaystyle 
\le \ 2 \int_\Omega{ \left[ (\nabla w_k(\psi))^t  A \, (\nabla w_k(\psi))
\, - \, 
(\nabla w_{k+1}(\psi))^t A \, (\nabla w_{k+1}(\psi)) \right] \, dx} .$ \\[2ex]
From this last equation we obtain
\begin{equation}
 ||T_l^k (\vphi)||_*^2 \ \le \
        2 \left( ||T_l^k (\psi)||_*^2 \ - \ ||T_l^{k+1} (\psi)||_*^2 \right) .
\label{GlMa3} \end{equation}
Another consequence of (\ref{GlMa1}) and (\ref{GlMa2}) is the inequality
$$  \int_\Omega{ (\nabla w_k)^t A (\nabla w_k) \, dx} \ - \
    \int_\Omega{ (\nabla w_{k+1})^t A (\nabla w_{k+1}) \, dx} \ \geq \ 0 ,  $$
\noindent  for every $k$, i.e. the sequence $\{ ||T_l^k (\psi)||_* \}$ does 
not increase. Now from (\ref{GlMa3}) follows
$$  \lim_{k\rightarrow\infty}{ \, ||T_l^k (\vphi)||_* } \ = \ 0 .  $$

        {\it ii)} For $\vphi \in H^{\me}_{00}(\Gamma_r)'$ define $w := 
L_n^l(\vphi)$ and $v := L_d^l \circ \gamma_d \circ L_n^l(\vphi)$. We claim 
that the inequality
\begin{equation}
  \langle T_l(\vphi), T_l(\vphi) \rangle_* \ \ \leq \ \ 
  \int_\Omega{ (\nabla v)^t  A \, (\nabla v) \, dx}
\label{GlMa4} \end{equation}
\noindent  holds. Indeed, as $T_l(\vphi) = \gamma_n v$ we have
\begin{eqnarray*}
\langle T_l(\vphi), T_l(\vphi) \rangle_*  & = &  \int_\Omega
  { (\nabla L_n^l(\gamma_n v))^t  A \, (\nabla L_n^l(\gamma_n v)) \, dx} \\
\!\!\!\!& = &\!\!\!\!  \int_{\Gamma_1\cup\Gamma_2}{ (L_n^l(\gamma_n v))_\nuA
             L_n^l(\gamma_n v) \, d\Gamma} \\ 
& & + \ \int_\Omega{ L_n^l(\gamma_n v) \, P(L_n^l(\gamma_n v)) \, dx} \\
& = & \int_{\Gamma_1\cup\Gamma_2}{v_\nuA \, L_n^l(\gamma_n v) \, d\Gamma}\ + \ 
      \int_\Omega{ P(v) \, L_n^l(\gamma_n v) \, dx} \\
& = &  \int_\Omega{ (\nabla v)^t  A \, (\nabla L_n^l(\gamma_n v)) \, dx} \\
&\leq& \left( \int_\Omega{ (\nabla v)^t  A \, (\nabla v) \, dx} \right)^{\me}
       \langle T_l(\vphi), T_l(\vphi) \rangle_*^{\me} ,
\end{eqnarray*}

\noindent  proving (\ref{GlMa4}). Now from the definition of $v$ and $w$ 
follows
\begin{eqnarray*}
\int_\Omega{ (\nabla v)^t  A \, (\nabla v) \, dx}  & = & 
                                  \int_{\Gamma_2}{ v_\nuA \, v \, d\Gamma} 
 = \int_{\Gamma_1\cup\Gamma_2}{ v_\nuA \, w \, d\Gamma}
 = \int_\Omega{ (\nabla v)^t  A \, (\nabla w) \, dx} \\
& \leq & \left( \int_\Omega{(\nabla v)^t  A \, (\nabla v) \, dx} \right)^{\me}
  \ \left( \int_\Omega{(\nabla w)^t  A \, (\nabla w) \, dx} \right)^{\me} .
\end{eqnarray*}
Putting all together we have
$$ ||T_l(\vphi)||_* \ \leq \ 
\left( \int_\Omega{(\nabla w)^t A \, (\nabla w) \, dx} \right)^{\me}
\ = \ \langle \vphi, \vphi \rangle_*^{\me} \ = \ ||\vphi||_* , $$
proving {\it (ii)}.
\qquad\end{proof}

        The next theorem guarantees the convergence of the iterative 
algorithm.

%=====================================%
\begin{theorem}  Let $T$ and $T_l$ be the operators defined in section 1.4. 
If we have consistent Cauchy--data $(f,g)$, then the sequence $\vphi_k = T^k 
\vphi_0$ converges to the Neumann--trace at $\Gamma_2$ of the solution of 
(CP) for every $\vphi_0 \in H^{\me}_{00}(\Gamma_2)'$.
\label{Tk_konverg} \end{theorem}

\begin{proof}
        It is enough to prove that $T^k \vphi_0$ converges to the fixed point 
$\overline{\vphi}$ of $T$. If we define $\eps_k = \vphi_k - \overline{\vphi}$ 
we have
\begin{eqnarray*}
\eps_{k+1} & = & \vphi_{k+1} - \overline{\vphi} \\
           & = & T(\vphi_k) - T(\overline{\vphi}) \\
           & = & T_l(\vphi_k)+z_{f,g} - T_l(\overline{\vphi})-z_{f,g} \\
           & = & T_l(\eps_k) \, .
\end{eqnarray*}
\noindent  Theorems~\ref{Tl_eigenschaft} and \ref{mazya} imply $\eps_k 
\rightarrow 0$.
\qquad\end{proof}

        The converse of Theorem~\ref{Tk_konverg} is valid, i.e. when the 
sequence $\vphi_k = T^k \vphi_0$ converges, the associated Cauchy problem is 
consistent and $\overline{\vphi} := \lim{\vphi_k}$ is the Neumann--trace of 
the problem's solution. This result can also be understood as an existence 
criterion for Cauchy problems.

%=====================================%
\begin{theorem}  Given the Cauchy data $(f,g) \in H^{\me}(\Gamma_1) \times 
H^{\me}_{00}(\Gamma_1)'$, we denote by $\{ \vphi_k \}$ the sequence 
generated by the iteration (IT). If $\{ \vphi_k \}$ converges in $H^{\me}_{00} 
(\Gamma_2)'$, the Cauchy problem (CP) has a solution $u$ in $H^1(\Omega;P)$ 
and $u_{\nuA|_{\Gamma_2}} = \lim_k{\vphi_k}$.
\label{exist_kriterium} \end{theorem}

\begin{proof}
        If we define $\overline{\vphi} = \lim_k{\vphi_k} \in H^{\me}_{00} 
(\Gamma_2)'$, we have
$$  T \, \overline{\vphi} \ = \ T(\lim_{k\rightarrow\infty}{\vphi_k}) \ = \ 
    \lim_{k\rightarrow\infty}{\vphi_{k+1}} \ = \ \overline{\vphi} \, .  $$
\noindent  Therefore $\overline{\vphi}$ is a fixed point of $T$. The argument 
of Remark~\ref{fix_punkt_bemerk} implies the existence of a solution for (CP)
 and the theorem follows.
\qquad\end{proof}

%----------------------------------------------------------------------------%
\subsection{A spectral property of $T_l$}

        Before going any further with the analysis of the iterative algorithm, 
we discuss an important spectral property of $T_l$. We have already proved in 
Theorem~\ref{Tl_eigenschaft} that $T_l$ is positive, self adjoint and it's 
spectrum belongs to $[0,1]$. Now we verify that 1 belongs to $\sigma(T_l)$.

%=====================================%
\begin{theorem}  Let $T_l$ be the operator defined in section 1.4. If there 
exists 
$(f,g) \in  H^{\me}(\Gamma_1) \times H^{\me}_{00}(\Gamma_1)'$ such that the 
Cauchy problem (CP) is inconsistent for the data $(f,g)$, then 1 belongs to 
the continuous spectrum $\sigma_c(T_l)$ of $T_l$.
\label{Tl_norm} \end{theorem}

\begin{proof}
      Let $\{ E_\lambda \}_{\lambda\in\R}$ be the spectral family for $T_l$ 
and denote by $I$ be the identity operator in $H^{\me}_{00}(\Gamma_2)'$. From 
Theorem~\ref{Tl_eigenschaft} we conclude that $E_\lambda = I$ for $\lambda 
\geq 1$. It's enough to prove that given $\delta \in (0,1)$ there exists an 
eigenvalue $\lambda_0$ of $T_l$ in the interval $(\delta,1)$.

      If this condition were not satisfied the point spectrum $\sigma_p(T_l)$ 
would be a subset of $[0,\delta]$ and $T_l$ would be contractive with norm 
$||T_l|| \leq \delta$. An immediate consequence of this is the convergence 
of the sequence $\vphi_k = T^k \vphi_0$, where $T = T_l + z_{f,g}$. Now 
Theorem~\ref{exist_kriterium} would imply the existence of a solution for 
the Cauchy problem (CP) with data $(f,g)$, contradicting the hypothesis of 
$(f,g)$ being inconsistent Cauchy data.
\qquad\end{proof}

%=====================================%
\begin{corollary}  From Theorems~\ref{Tl_eigenschaft} and \ref{Tl_norm} 
follows\, $||T_l|| = 1$. \hfill \qquad
\end{corollary}

%----------------------------------------------------------------------------%
\subsection{Error estimation}

        For simplicity we investigate in this section the iteration (IT) for 
the operator $P = -\Delta$ in two special domains. Analog results can be 
obtained for general operators of the form (\ref{L-definition}) every time 
the spectral decomposition of the operator is known.

        In the first problem we take $\Omega = (-\pi,\pi)\times(-\pi,\pi)$, 
$\Gamma_1 = \{ (x,0);\ x \in (-\pi,\pi) \}$, $\Gamma_2 = \{ (x,\pi);\ x \in 
(-\pi,\pi) \}$ and want to solve the problem

\medskip \noindent
$ (CP\ 1)            \hfill \left\{  \begin{array}{rl}
                             \Delta u = 0     ,& \mbox{in} \ \Omega \\
                             u = f            ,& \mbox{at} \ \Gamma_1 \\
                             u_\nu = g        ,& \mbox{at} \ \Gamma_1 \\
                             u(x,\pm\pi) = 0  ,& x \in (-\pi,\pi)
                     \end{array} \right.  \hfill  $
\medskip

\noindent  In the second problem $\Omega$ is the ring centered at the origin 
with inner and outer radius respectively $r_0$ and 1, $\Gamma_1 = \{ 
(1,\theta);\ \theta \in (-\pi,\pi) \}$, $\Gamma_2 = \{ (r_0,\theta);\ \theta 
\in (-\pi,\pi) \}$. The problem to be solved is

\medskip \noindent
$ (CP\ 2)             \hfill \left\{  \begin{array}{rl}
                            \Delta u = 0 ,& \mbox{in}\ \Omega    \\
                            u = f        ,& \mbox{at}\ \Gamma_1 \\
                            u_\nu = g    ,& \mbox{at}\ \Gamma_1 \\
                      \end{array} \right.   \hfill  $
\medskip

        As we are working on special domains, it is possible to describe the 
action of the operator $T_l$ explicitly. If $\vphi_0 = \sum{\vphi_{0,j} \, 
\sin(jy)}$ is given in the Sobolev space of periodic functions%
\footnote{For $s\in\R$ one defines $ H^s_{\rm per}((-\pi,\pi)) := 
\{ \vphi(y) = \sum\limits_{k\in\Z}{c_k \ e^{iky}}\ |\ \sum\limits_{k\in\Z} 
{(1+k^2)^{s} c_k^2} \ < \ \infty \}$.}
$H^{\mme}_{\rm per}(\Gamma_2)$, we have for (CP 1)
\begin{equation}
(T_l^k \, \vphi_0)(x) \ = \ 
              \sum_{j=1}^{\infty}{ \lambda_j^{2k} \, \vphi_{0,j}\, \sin(jx) },
\label{T_l_potenz} \end{equation}
\noindent  where $\lambda_j = sinh(2j\pi)/cosh(2j\pi),\ j\in\N$. As we intend 
to measure how fast the error $\eps_k := \vphi_k - \overline{\vphi}$ 
converges to zero, we deduce from (\ref{T_l_potenz}) and from the equality 
$\eps_{k+1} = T_l \, \eps_k$ the estimate
\begin{equation}
  ||\eps_k||_{H^{\mme}_{\rm per}(\Gamma_2)}^2 \ \leq \ 
  \sum_{j\geq 1}{j^{-1}\, \left( \lambda_j^k\,\eps_{0,j} \right)^2} \, .
\label{abschaetz_eigenwerte} \end{equation}
        If the initial error $\eps_0$ has the nice property of consisting only 
of the lower frequencies $j \leq J$, equation~(\ref{abschaetz_eigenwerte}) 
simplifies to
$$ ||\eps_k||_{H^{\mme}_{\rm per}(\Gamma_2)}^2 \ \leq \ 
        \lambda_j^{2k} \, ||\eps_0||_{H^{\mme}_{\rm per}(\Gamma_2)}^2 \, .  $$
\noindent  In the very special case $J=1$ and $\eps_0 = \eps_{0,j} \sin(x)$ 
one calculates for $k=10^5$ the power of the first eigenvalue 
$\lambda_1^{2k} = 0.061$. Therefore we must evaluate $10^5$ iteration steps 
to reduce the error to 6\% of the initial error.

        Next we analyze a more realistic situation, in which the initial 
error $\eps_0 = \vphi_0 - \overline{\vphi}$ has more regularity than a 
$H^{\mme}_{\rm per}(\Gamma_2)$ distribution. We assume that there exists a 
monotone sequence of positive real numbers $\{ c_j \}$ such that
$$ \lim_{j \rightarrow\infty}{c_j} \ = \ \infty \ \ \ \mbox{and} \ \ \
   \sum_{j\geq 1}{ j^{-1} \, c_j^2 \, \eps_{0,j}^2 }\ =\ M\ < \ \infty \, . $$
In this case the error at the $\rm k^{\rm th}$--iteration step can be 
estimated by
\begin{eqnarray}
||\eps_k||_{H^{\mme}_{\rm per}(\Gamma_2)}^2  & \leq &
     \lambda_J^{2k} \, \left( \frac{c_J}{c_1} \right)^2 \, \sum_{j\leq J}
     { j^{-1} \, \eps_{0,j}^2 } \ + \ \frac{1}{c_J^2} \, \sum_{j>J}
     { j^{-1} \, c_j^2 \, \lambda_j^{2k} \eps_{0,j}^2}  \nonumber \\
& \leq & \lambda_J^{2k} \, \left( \frac{c_J}{c_1} \right)^2 \,
         ||\eps_0||_{H^{\mme}_{\rm per}(\Gamma_2)}^2 \ + \ \frac{M}{c_J^2} .
                                                     \label{abschaetz_gewicht}
\end{eqnarray}
        For the Cauchy problem (CP 2) we have an analogous result. If the 
iteration is again formulated at $H^{\mme}_{\rm per}(\Gamma_2)$ one obtains 
for the operator $T_l$ the eigenfunctions $\sin(j\theta)$, $cos(j\theta)$ 
with corresponding eigenvalues
$$ \lambda_j \ = \ [(r_0^{-(j+1)} - r_0^{j-1}) (r_0^{-j} - r_0^{j})]\ /\ 
                   [(r_0^{-(j+1)} + r_0^{j-1}) (r_0^{-j} + r_0^j)] \, .  $$
        In Figure~\ref{fig1} we show a qualitative comparison between the 
eigenvalues of $T_l$ in the different domains considered in this section.

%-------- Figure 1
\begin{figure}
\epsfysize4cm \centerline{ \epsfbox{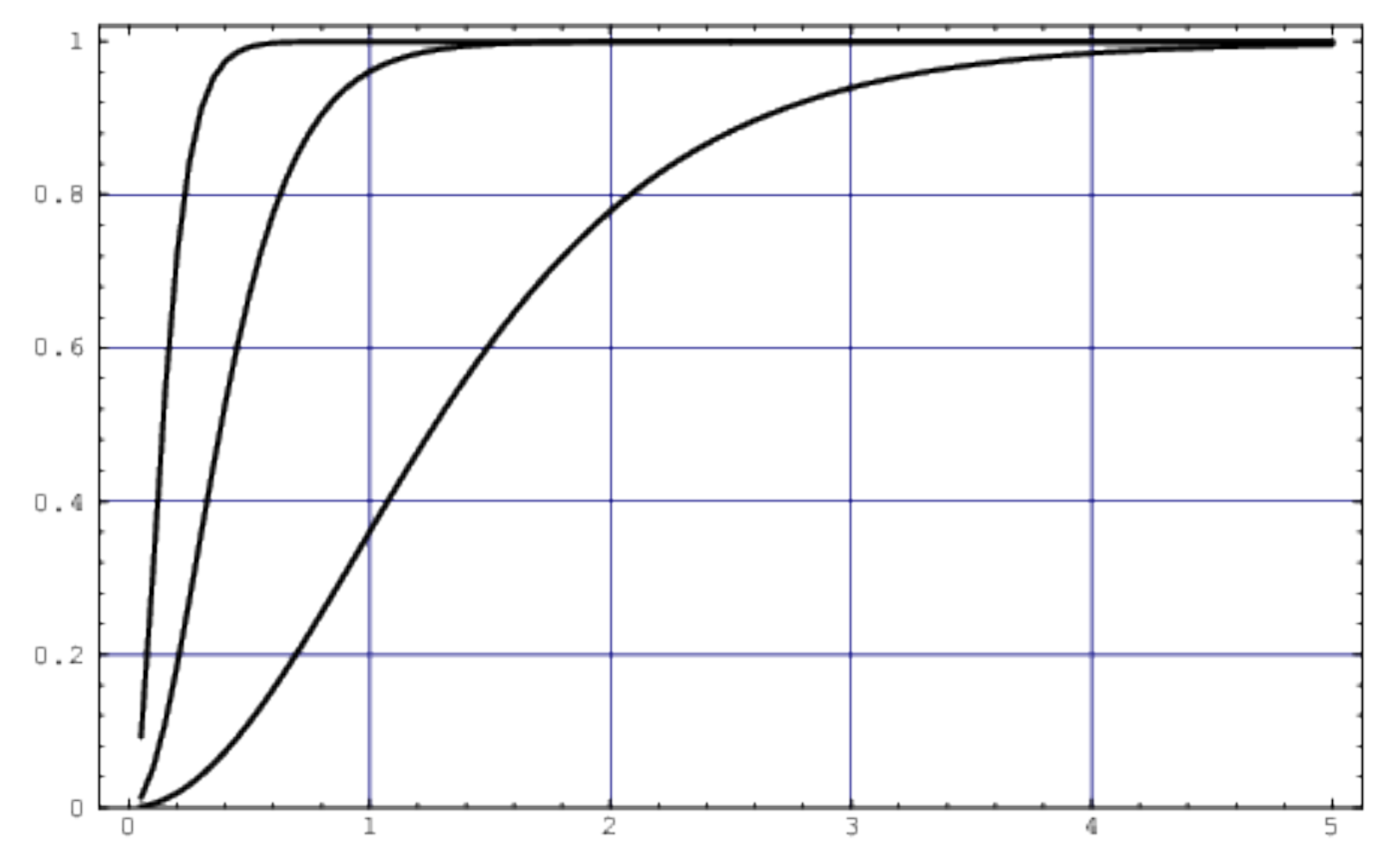} }
\vskip-4cm \unitlength1cm
\begin{center}
\begin{picture}(10,4)
\put(2.4,3.5){\small A} \put(3.2,2.8){\small B} \put(4.1,2.1){\small C}
\put(6.5,2.8){\small A} \put(6.5,2.1){\small B} \put(6.5,1.3){\small C}
\put(7,1.4){\footnotesize Ring} \put(7,1.1){\footnotesize $r_0 = 0.5$} 
\put(7,2.1){\footnotesize Ring} \put(7,1.8){\footnotesize $r_0 = 0.1$}
\put(7,2.8){\footnotesize Square}
\put(5.2,-0.1){$j$} \put(1.3,2){$\lambda_j$}
\end{picture}
\caption{Eigenvalues of $T_l$ in different domains \label{fig1}}
\end{center}
\end{figure}
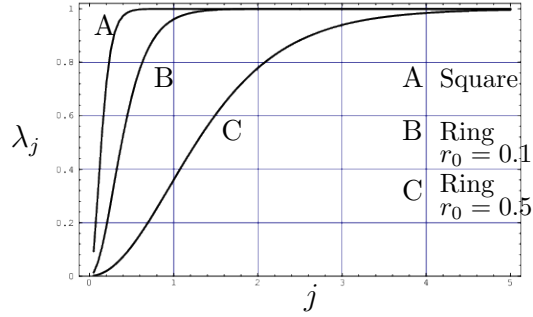
%------------------

%----------------------------------------------------------------------------%
\subsection{Regularization}

       The objective of regularizing the iteration (IT) is to choose an 
operator $T_{\rm reg}$ such that the regularized sequence $\tilde{\vphi_k} 
:= T_{\rm reg}^k \, \vphi_0 + \sum_{j<k} T_{\rm reg}^j \, z_{f,g}$ converges 
faster than the original sequence $\vphi_k := T_l^k \, \vphi_0 + \sum_{j<k} 
T_l^j \, z_{f,g}$. We also have to assure that the difference 
$\|\lim{\tilde{\vphi_k}} - \lim{\vphi_k}\|$ remains small.

        We start with the {\em a priori} assumption that the Cauchy data 
of (CP) satisfies $(f,g) \in H^{r}(\Gamma_1) \times H^{\me}_{00}(\Gamma_1)'$, 
with $r > \me$.

Given the measured data $(f_\eps,g_\eps)$ in $L^2(\Gamma_1) \times 
H^{\me}_{00}(\Gamma_1)'$, we claim that using a smoothing operator 
$S: L^2 \to H^{\me}$ it is possible to generate a $\tilde{f_\eps} := 
S f_\eps \in H^{\me}$ satisfying $\| f - \tilde{f_\eps} \|_{\me} \leq \eps'$. 
Indeed this is a consequence of 

\begin{lemma} \label{lemma-fehler-glaetung}
        Let $f \in H^r$, $r > s > 0$. There exists a smoothing operator 
$S: L^2 \to H^s$ and a positive function $\gamma$ with $\lim_{x\downarrow 0} 
\gamma(x) = 0$, such that for $\eps > 0$ and $f_\eps \in L^2$ with 
$\|f - f_\eps\|_{L^2} \leq \eps$, we have $\|f - S f_\eps\|_s \le 
\gamma(\eps)$.
\end{lemma}
\begin{proof}
This lemma describes a standard procedure in inverse problems. A complete 
proof can be found in [BaLe].
\end{proof}

        After smoothing the data $f_\eps$, we obtain a corresponding 
$z_\eps \in {H^{\me}(\Gamma_l)'}$ such that $|| z_{f,g} - z_\eps ||_{H^{\me} 
(\Gamma_l)'} < \eps$.

        We analyze the choice of two different regularization strategies. 
The first one is a cut--off method, where we consider only the eigenvalues 
of $T_l$ lower than $1 - 1/n$.%
\footnote{We may suppose the real numbers $1 - 1/n$ are not eigenvalues of 
$T_l$.}
In the second method we use powers of $T_l$ to define $T_{\rm reg}$. For 
$n \geq 2$ we set
\begin{equation}
      A_n \ := \ \int_{0}^{1-\frac{1}{n}}{\, \lambda \, d E_\lambda}
\ \ \ \ \ \ \mbox{and} \ \ \ \ \ \
      B_n \ := \ \int_{0}^{1}{\, (\lambda - \lambda^n) \, d E_\lambda} \, ,%
\label{An_Bn_def} \end{equation}
\noindent  where $E_\lambda$ is the spectral family of $T_l$. Both operators 
$A_n$ and $B_n$ are positive, self adjoint and contractive. Let 
$T_{\rm reg}^{(n)}$ represent one of the families defined in (\ref{An_Bn_def}) 
and define $\overline{\vphi}$ and $\vphi^{(n)}$ as the fixed points of 
$$      \overline{\vphi} \ = \ T_l \, \overline{\vphi} + z_{f,g}
        \ \ \ \ \ \ \ \ {\rm and} \ \ \ \ \ \ \ \ 
        \vphi^{(n)} \ = \ T_{\rm reg}^{(n)} \, \vphi^{(n)} + z_\eps  $$
respectively. $\vphi^{(n)}$ will exist, as $T_{\rm reg}^{(n)}$ is 
contractive. We have now
\begin{eqnarray}
||\varphi^{(n)} \, - \, \overline{\vphi}||  & = &
       || T_l \, \overline{\vphi} + z_{f,g} - T_{\rm reg}^{(n)} \, \vphi^{(n)}
       - z_\eps ||  \nonumber  \\
& = & || T_{\rm reg}^{(n)}(\vphi^{(n)} - \overline{\varphi}) +
      \left( T_{\rm reg}^{(n)} - T_l \right) \, \overline{\vphi} -
      z_{f,g} + z_\eps ||  \nonumber  \\
& \leq & ||\left( I - T_{\rm reg}^{(n)} \right)^{-1} \,
         \left( T_{\rm reg}^{(n)} - T_l \right) \, \overline{\vphi}|| \ + \
         \eps \, ||\left( I - T_{\rm reg}^{(n)} \right)^{-1}|| .
                                            \label{abschaetz_regulariz_fehler}
\end{eqnarray}
        In next theorem we analyze the estimate 
(\ref{abschaetz_regulariz_fehler}) for the operators $A_n$ and $B_n$.

%=====================================%
\begin{theorem}  If we define the family of operators $T_{\rm reg}^{(n)}$ 
using one of the families in (\ref{An_Bn_def}) we have
\begin{equation}
||\left( I - T_{\rm reg}^{(n)} \right)^{-1} \,
    \left( T_{\rm reg}^{(n)} - T_l \right) \, \overline{\vphi}|| \to 0
\ \ \ {\rm and} \ \ \
    ||\left( I - T_{\rm reg}^{(n)} \right)^{-1}|| \to \infty .
\label{abschaetz_regulariz_terme} \end{equation}
\label{diskrep_princ} \end{theorem}

\begin{proof}
        {\it i)}  We analyze the case $T_{\rm reg}^{(n)} = A_n$ first. 
From the spectral decomposition of $T_l$ follows
$$ (I - A_n) \, = \, \int_{0}^{1-\frac{1}{n}}{\, (1 - \lambda) \, d E_\lambda}
                  + \int_{1-\frac{1}{n}}^{1}{\, d E_\lambda}
\ \ \ \ \ \ {\rm and} \ \ \ \ \ \ 
(A_n - T_l) \, = \, -\int_{1-\frac{1}{n}}^{1}{\, \lambda \, d E_\lambda} . $$
\noindent  Now these equalities imply $||(I - A_n)\, \vphi|| \geq 1/n \, 
||\vphi||$, and the operator $(I - A_n)$ has an inverse. We also know that 
$(I - A_n)$ is the identity operator on $Rg(A_n - T_l)$. From this follows
\begin{equation} (I - A_n)^{-1}(A_n - T_l) \ = \ 
                - \int_{1-\frac{1}{n}}^{1}{\, \lambda \, d E_\lambda} ,
\label{extra_equality} \end{equation}
\noindent  and we can estimate the first term in 
(\ref{abschaetz_regulariz_terme}) by
$$  ||(I - A_n)^{-1}(A_n - T_l) \, \overline{\vphi}||^2 \ \leq \
    \int_{1-\frac{1}{n}}^{1}{\,  d \langle E_\lambda \, \overline{\vphi} ,
                                       \overline{\vphi} \rangle} \ \to \ 0  $$
\noindent  for $n \rightarrow \infty$. For the second term in 
(\ref{abschaetz_regulariz_terme}) we use the identity
\begin{equation}
 ||(I - A_n)^{-1}|| \ = \ (1 - \Lambda(n))^{-1} \ \rightarrow \ \infty
\label{extra_equality2} \end{equation}
\noindent   for $n \rightarrow \infty$, where $\Lambda(n)$ is the largest 
eigenvalue of $T_l$, which is smaller than $(1 - 1/n)$.

        {\it ii)}  For the case $T_{\rm reg}^{(n)} = B_n$ we have $(B_n - 
T_l) \ = -\, T_l^n$ and from the spectral decomposition of $T_l$ follows
$$  (I - B_n)^{-1}(B_n - T_l) \ = \ \int_{0}^{1}{
          \frac{\lambda^n}{(1 + \lambda^n - \lambda)} d E_\lambda} .  $$
\noindent  If we define the functions $\mu(n) := n^{(\frac{1}{1-n})}$ \ and \ 
$\delta(n) := 1 - \mu(n)$ we can decompose the operator $(I - B_n)^{-1} 
(B_n - T_l)$ in ${\cal Q}_n + {\cal R}_n$ where
$$ {\cal Q}_n \ = \ \int_{0}^{1-\delta(n)}{\, \frac{\lambda^n}{(1 + \lambda^n 
                 - \lambda)} \,  d E_\lambda}
\ \ \ \ \ \ \ {\rm and} \ \ \ \ \ \ \ 
   {\cal R}_n \ = \ \int_{1-\delta(n)}^{1}{\, \frac{\lambda^n}{(1 + \lambda^n 
                 - \lambda)} \, d E_\lambda} \, .                           $$
\noindent  From the convergence $\mu^n(n) \, (1 + \mu^n(n) - \mu(n))^{-1} 
\rightarrow 0$ for $n \rightarrow \infty$ follows $\lim {||{\cal Q}_n||} = 
0$. The convergence $||{\cal R}_n|| \rightarrow 0$ follows from inequality 
$0 < \lambda^n \, (1 + \lambda^n - \lambda)^{-1} \leq 1,\ \forall\ \lambda 
\in [1-\delta(n),1], \ \forall n \geq 2$. With this we have proved
$$    \lim_{n \rightarrow \infty}{\, ||(I - B_n)^{-1}(B_n - T_l) \, 
                                            \overline{\vphi}||} \ = \ 0.    $$
        To obtain the second limit in (\ref{abschaetz_regulariz_terme}) we 
deduce from the spectral decomposition of $(I-B_n)^{-1}$ the equality
$$  ||(I - B_n)^{-1}|| \ = \ (1 + \Upsilon^n(n) - \Upsilon(n))^{-1} \
                                                      \rightarrow \ \infty  $$
\noindent  for $n \rightarrow \infty$, where $\Upsilon(n)$ is the largest 
eigenvalue of $T_l$, which is smaller than $\mu(n)$.
\qquad\end{proof}

        Our next step it to use {\em a priori} information about the solution 
$\overline{\vphi}$ of the fixed point equation $T\, \vphi = \vphi$ in order 
to find an optimal regularization strategy. Let's suppose there exists a 
function  $G$ with
\begin{equation}
\left\{ \renewcommand{\arraystretch}{2} \begin{array}{l}
        G : [0,1) \mapsto \R^+ \ \mbox{is continuous and monotone increasing}; \\
        {\displaystyle
        \lim_{\lambda \rightarrow 1^-} {G(\lambda)} \ = \ \infty; } \\
        {\displaystyle
        \int_{0}^{1} {G^2(\lambda) \ d \langle E_\lambda\, \overline{\varphi},
                \overline{\varphi} \rangle} \ = M^2 \ < \ \infty. }
        \end{array} \right.
\label{regular_annahme} \end{equation}
\noindent  In the next theorem we analyze how the regularity condition in 
(\ref{regular_annahme}) can be used to balance the approximation and 
regularization errors.

%=====================================%
\begin{theorem}  Let $G$ be a function which satisfy (\ref{regular_annahme}) 
and $\delta$, $\mu$, $\Lambda$, $\Upsilon$ be the functions used in the proof 
of Theorem~\ref{diskrep_princ}.  For the two regularization strategies 
in (\ref{An_Bn_def}) there exists $n_{\rm opt} \in \N$, such that 
$$ ||\vphi^{(n_{\rm opt})} \, - \, \overline{\vphi}|| \ \leq \ 
                                ||\vphi^{(n)} \, - \, \overline{\vphi}|| ,  $$
\noindent  for every $n \in \N$. Further $n_{\rm opt}$ is obtained by solving 
the minimization problem
$$  \min_{n \geq 2} \ \left\{ \frac{M}{G(1-\frac{1}{n})} \ + \
                             \frac{\eps}{(1 - \Lambda(n))} \right\}  $$
\noindent  for the regularization strategy using $A_n$. For the $B_n$ 
regularization strategy $n_{\rm opt}$ is obtained as the solution of the 
minimization problem
$$  \min_{n \geq 2} \ \left\{ \left(
\frac{\mu^n(n) \, ||\overline{\vphi}||} {1 + \mu^n(n) - \mu(n)} \, + \,
\frac{M}{G(\mu(n))}             \right) \ + \
\frac{\eps}{(1 + \Upsilon^n(n) - \Upsilon(n))} \right\} .             $$
\label{regul_strateg_opt} \end{theorem}  

\begin{proof}
        We show here only the proof for the regularization strategy $A_n$, the 
second case being analog. From (\ref{extra_equality}) and 
(\ref{regular_annahme}) follows
\begin{eqnarray*}
  ||(I - A_n)^{-1}(A_n - T_l) \, \overline{\vphi}||^2 & = &
  \int_{1-\frac{1}{n}}^{1} {\lambda^2 \ d \langle E_\lambda \, 
                            \overline{\vphi} , \overline{\vphi} \rangle} \\
& \leq & \frac{1}{G^2(1-\frac{1}{n})} \,
  \int_{1-\frac{1}{n}}^{1} {G^2(\lambda) \ d \langle E_\lambda \, 
                            \overline{\varphi} , \overline{\vphi} \rangle} \\
& \leq & \frac{M^2}{G^2(1-\frac{1}{n})}.
\end{eqnarray*}
\noindent  Now inequality (\ref{abschaetz_regulariz_fehler}) and 
(\ref{extra_equality2}) imply
$$  ||\vphi^{(n)} \, - \, \overline{\vphi}|| \ \leq \
           \frac{M}{G(1-\frac{1}{n})} \ + \ \frac{\eps}{(1 - \Lambda(n))},  $$
\noindent   and the assertion follows.
\qquad\end{proof}

\begin{remark}  One can interpret the regularity condition in 
(\ref{regular_annahme}) as follows: With the aid of the function $G$ one can 
define the unbounded operator
$$         {\cal G} \ = \ \int_{0}^{1}{\, G(\lambda) \, d E_\lambda}        $$
\noindent   on $H^{\me}_{00}(\Gamma_2)'$. The existence of the integral in 
(\ref{regular_annahme}) is equivalent to the assumption that $\overline 
{\vphi}$ belongs to $D({\cal G})$, the domain of ${\cal G}$ defined by
$$  D({\cal G}) \ := \ \left\{ \varphi \in H^{\me}_{00}(\Gamma_2)' \ | \ 
                 {\cal G}(\varphi) \in H^{\me}_{00}(\Gamma_2)' \right\} .   $$
\end{remark}

\section{Numerical experiments}

        In this section we present some results obtained by the numerical 
implementation of the iterative algorithm. In the first two examples in 
sections~3.1 and 3.2 respectively we solve linear consistent problems 
in a square and in a annular domain. In section~3.3 we exhibit a linear 
inconsistent problem and in section~3.4 we analyze a non linear 
consistent problem.

        The computation was performed on the IBM--RISC/6000 machines at 
the Federal University of Santa Catarina. The elliptic mixed boundary 
value problems that appear in the iteration were solved using the PLTMG 
package (see [Ban]).

%============================================================================%
\subsection{A linear problem in a square domain}

        In this example we take $\Omega = (0,1)\times(0,3/4)$ and decompose 
the boundary $\partial\Omega$ in  $\Gamma_1 \cup\Gamma_2 \cup \Gamma_3 \cup 
\Gamma_4$, where
$$         \Gamma_1 := \{ (x,0) ; x \in (0,1) \}\, , \ \ \ \ \ \ 
           \Gamma_2 := \{ (x, 3/4) ; x \in (0,1) \}\, ,       $$
$$        \Gamma_3 := \{ (0,y) ; y \in (0,3/4) \}\, , \ \ \ \,
          \Gamma_4 := \{ (1,y) ; y \in (0,3/4) \}\, .    $$
\noindent  Given the Cauchy data\, $f(x) = \sin(\pi x)$\, and\, $g(x) = 0$\, 
at $\Gamma_1$ we reconstruct the (Dirichlet) trace at $\Gamma_2$ of the 
solution of following Cauchy Problem:
$$  \left\{ \begin{array}{ccl}
            \Delta \, u  & = & 0 \ ,\ \mbox{ in } \Omega \\
            u            & = & f \ ,\ \mbox{ at } \Gamma_1 \\
            u_{\nu}      & = & g \ ,\ \mbox{ at } \Gamma_1 \\
            u            & = & 0 \ ,\ \mbox{ at } \Gamma_3 \cup \Gamma_4
            \end{array} \right. \, .                                        $$
\noindent  The exact solution of this Cauchy Problem is $u^\ast(x,y) = 
\cosh(\pi y) \, \sin(\pi x)$.

         Each mixed problem is solved using an uniform mesh with 262\,913 
nodes and linear elements. The trace at $\Gamma_2$ of the sequence generated 
by (IT) is shown (solid line) after 10, 25, 50 and 100 steps at 
Figure~\ref{fig2}. The dotted line represent the trace of the exact solution 
$u^\ast$ at $\Gamma_2$.

        As a stopping criterion we choose $||\psi_{k+1} - \psi_{k}||_{\infty; 
\Gamma_2} \leq 10^{-3}$. In this example the iterative sequence converges 
extremely fast. We observe a slower rate of convergence when the mesh is 
refined, but the approximation obtained with the same stopping criterion is 
more accurate.

%-------- Figure 2
\begin{figure}
\centerline{\epsfxsize3.3cm \epsfysize4cm \epsfbox{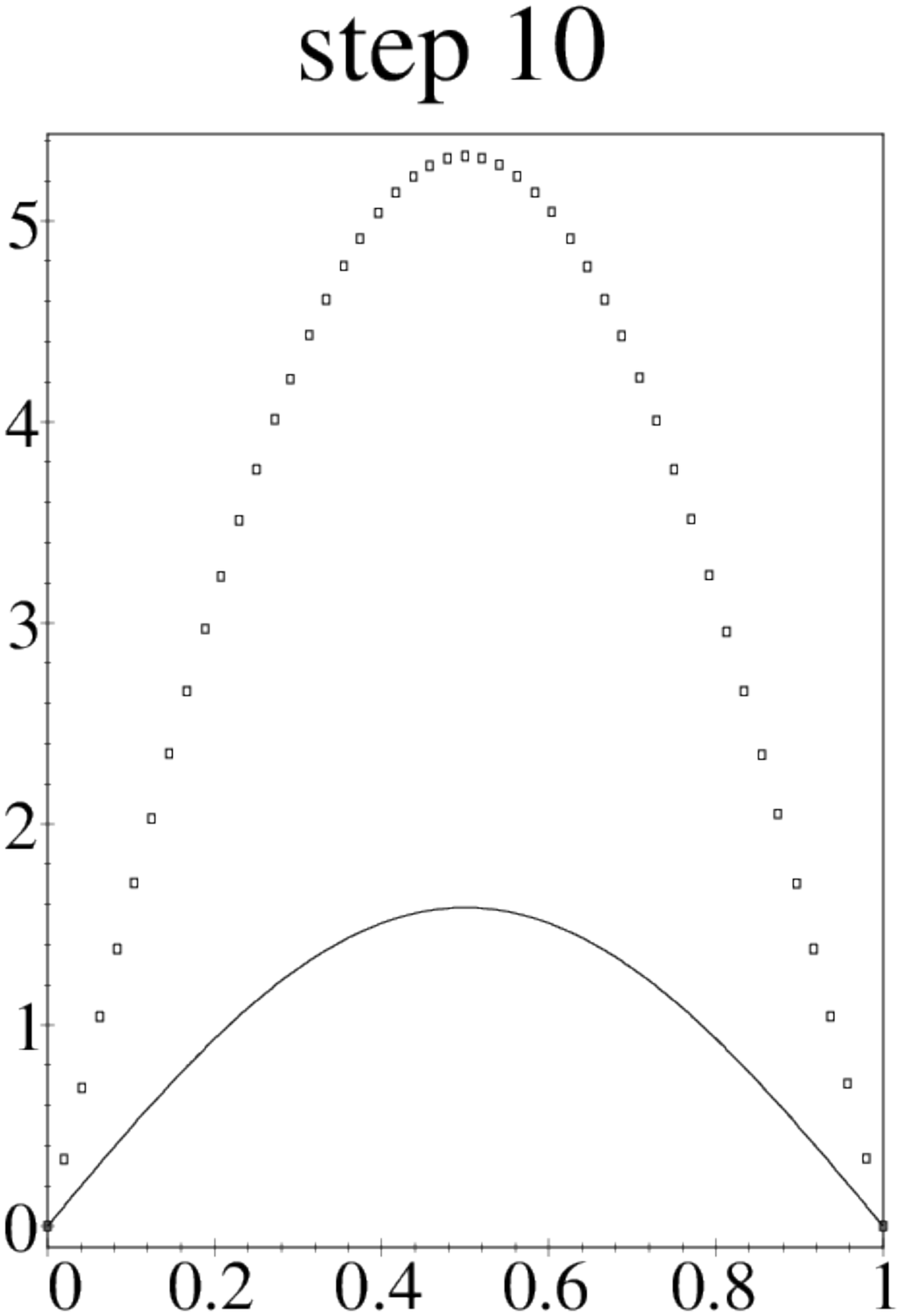}
            \epsfxsize3.3cm \epsfysize4cm \epsfbox{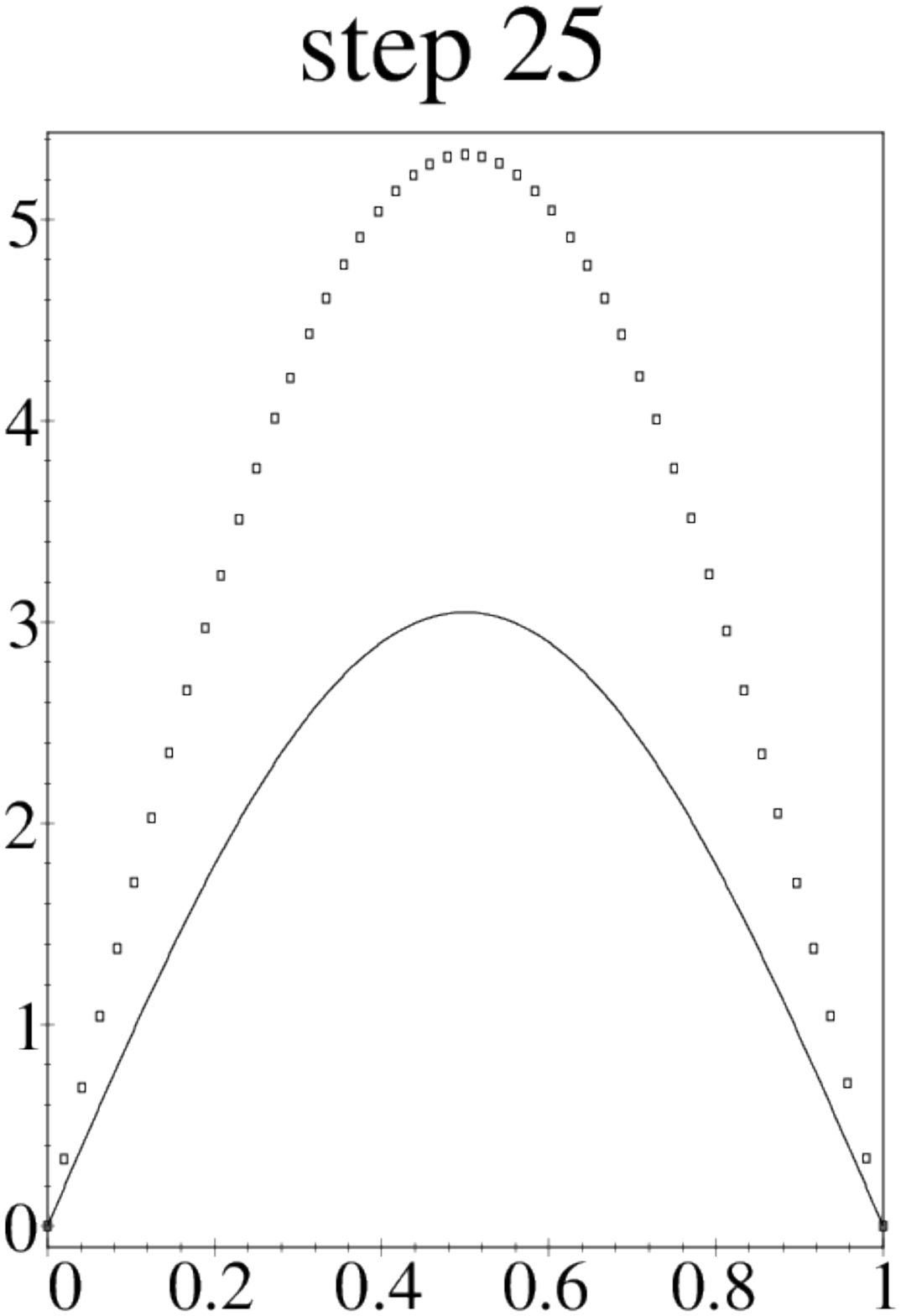}
            \epsfxsize3.3cm \epsfysize4cm \epsfbox{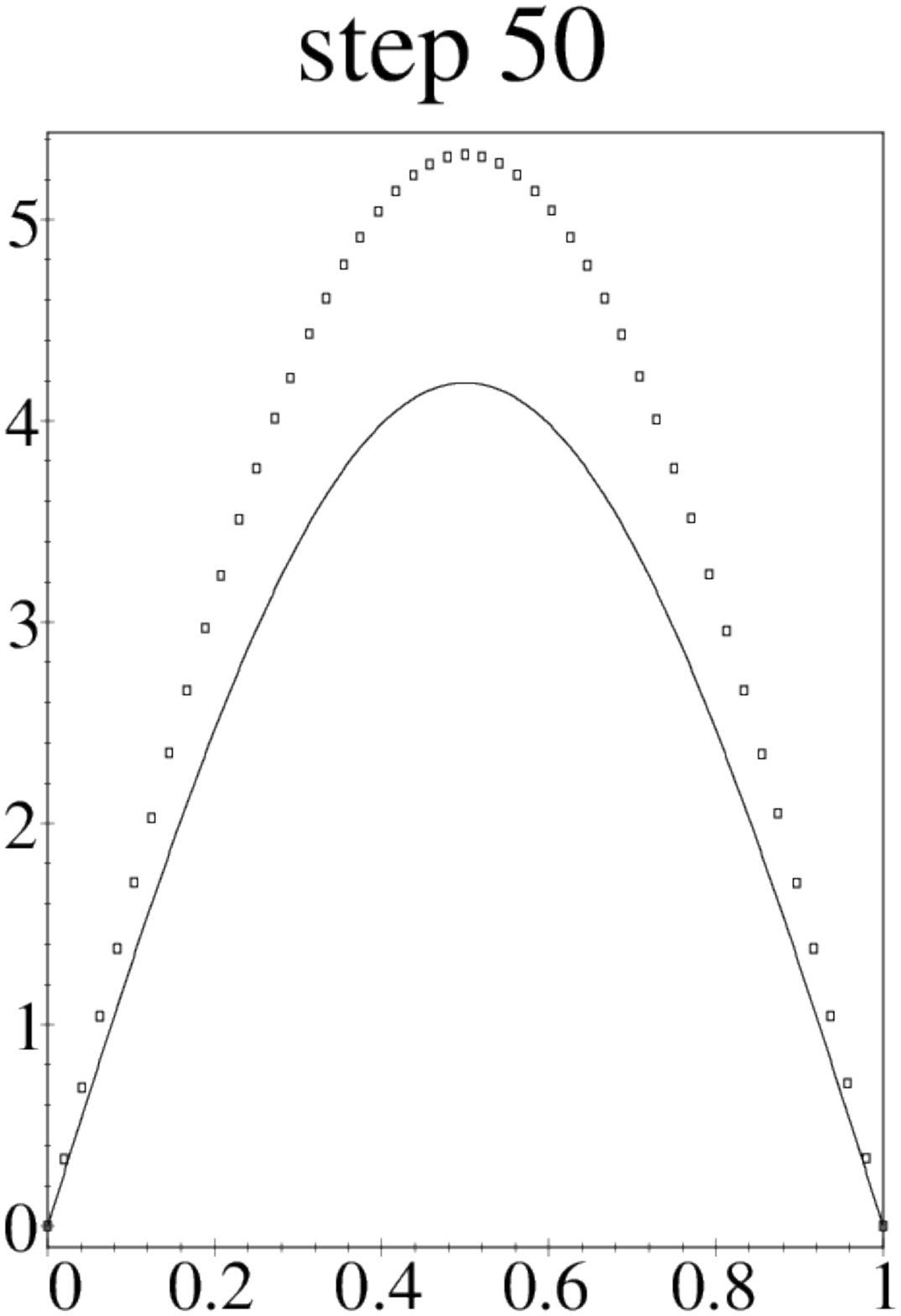}
            \epsfxsize3.3cm \epsfysize4cm \epsfbox{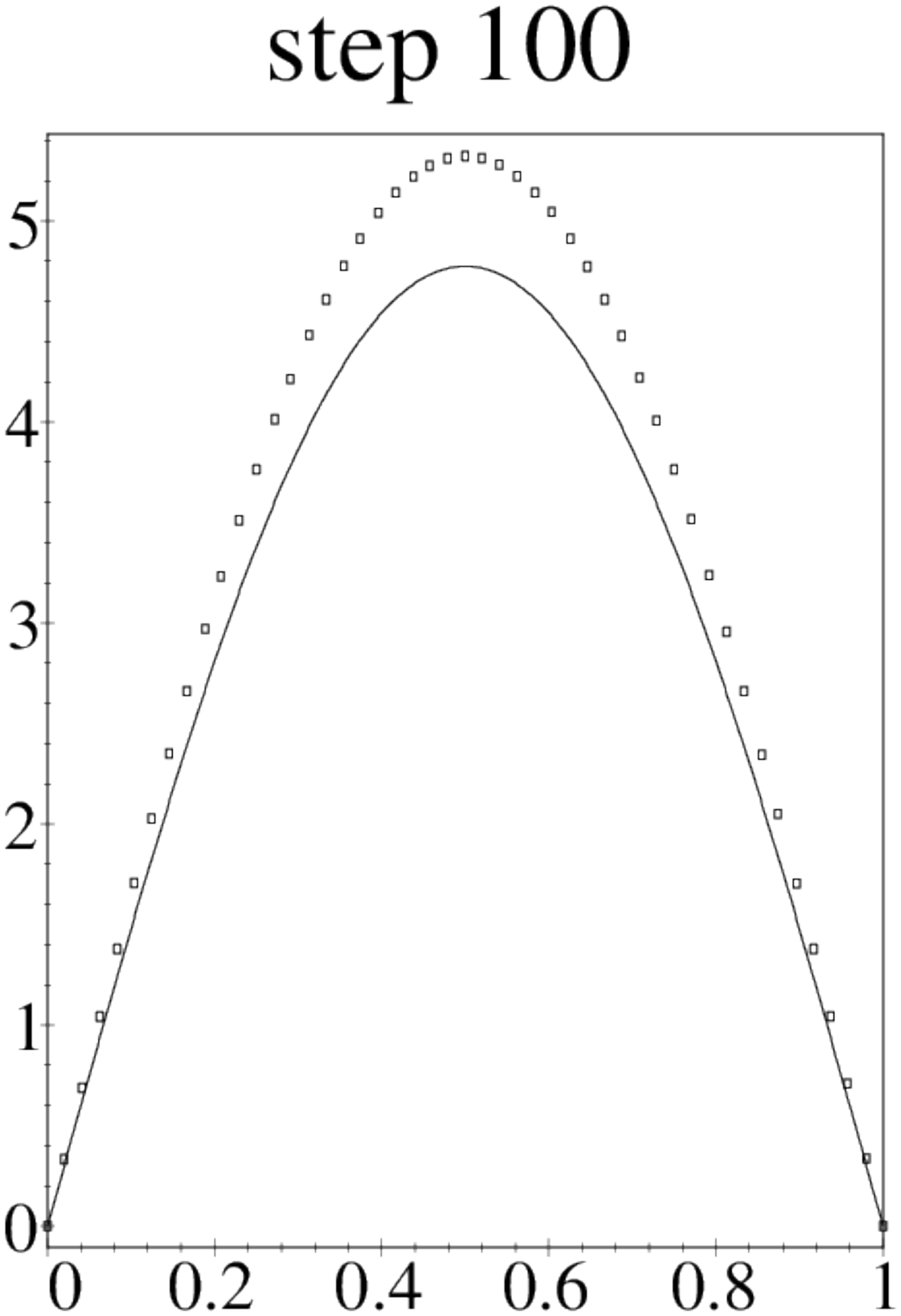} }
\caption{Iterated sequence at the unknown boundary for a linear Cauchy 
problem at the domain $\Omega = (0,1) \times (0,3/4)$ \label{fig2}}
\end{figure}
%------------------

%============================================================================%
\subsection{A linear problem in an annular domain}

        In this second example $\Omega$ is an annulus centered at the 
origin with inner and outer radius respectively 1 and 7. Given the Cauchy 
data\, $f(\theta) = \sin(\theta)$\, and\, $g(\theta) = 0$\, at the inner 
boundary $\Gamma_1$, we reconstruct at the outer boundary $\Gamma_2$ the 
trace of the solution of following problem:
$$  \left\{ \begin{array}{ccl}
            \Delta \, u  & = & 0 \ ,\ \mbox{ in } \Omega \\
            u            & = & f \ ,\ \mbox{ at } \Gamma_1 \\
            u_{\nu}      & = & g \ ,\ \mbox{ at } \Gamma_1 \\
            \end{array} \right. \, .                                        $$
\noindent  The exact solution of this problem is $u^\ast(x,y) = (r+1/r) 
\sin(\theta)/2$. We used a finite element mesh with 61\,824 nodes, linear 
elements and the stopping criterion $||\psi_{k+1} - \psi_{k}||_{\infty; 
\Gamma_2} \leq 10^{-4}$. The dotted line in Figure~\ref{fig3} represents 
the exact solution (note that the $x$--axis is parameterized from zero to 
$2\pi$) and the solid line represents the sequence $\psi_k$ generated by (IT).

%-------- Figure 3
\begin{figure}
\centerline{ \epsfxsize3.3cm \epsfysize4cm \epsfbox{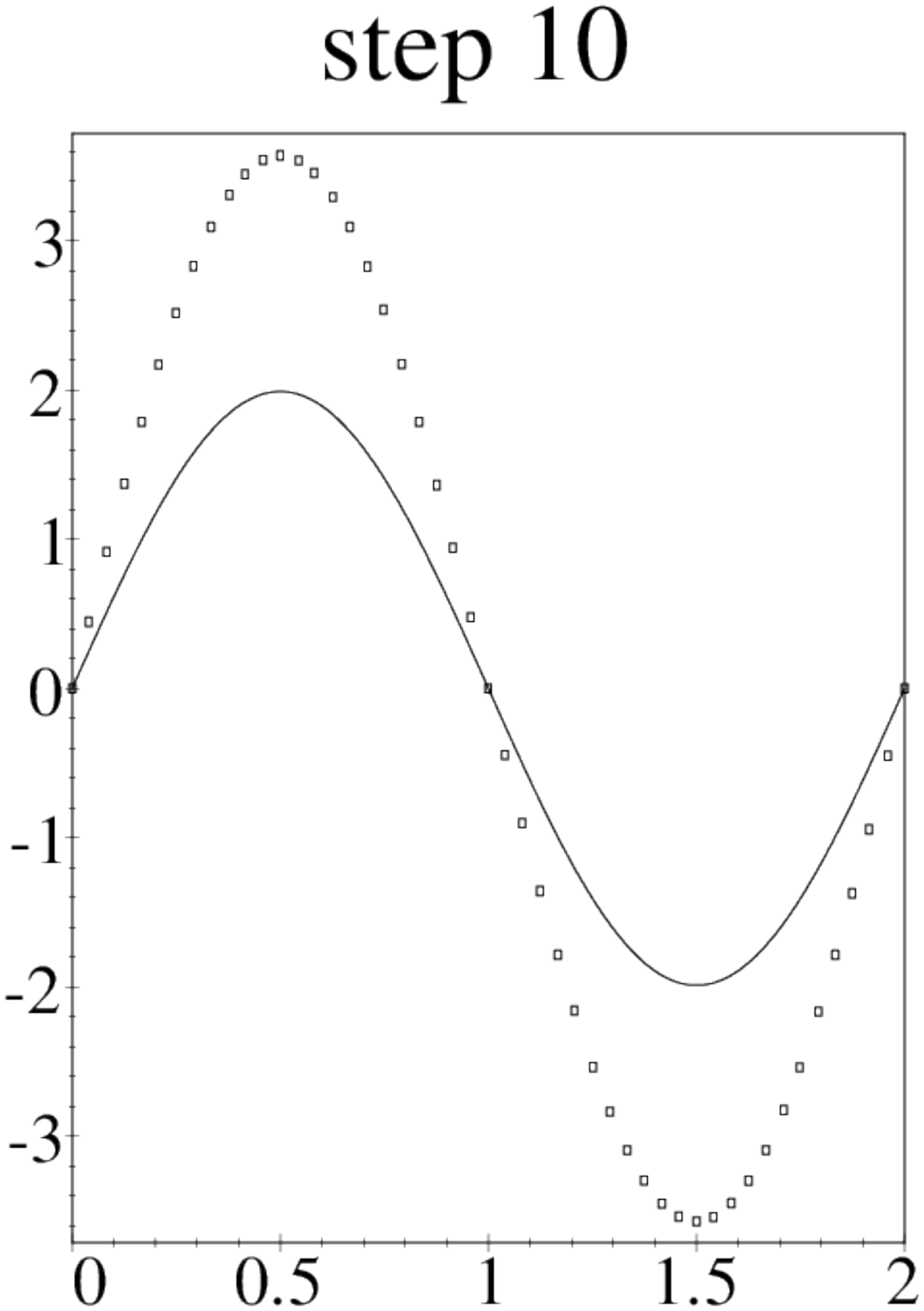}
             \epsfxsize3.3cm \epsfysize4cm \epsfbox{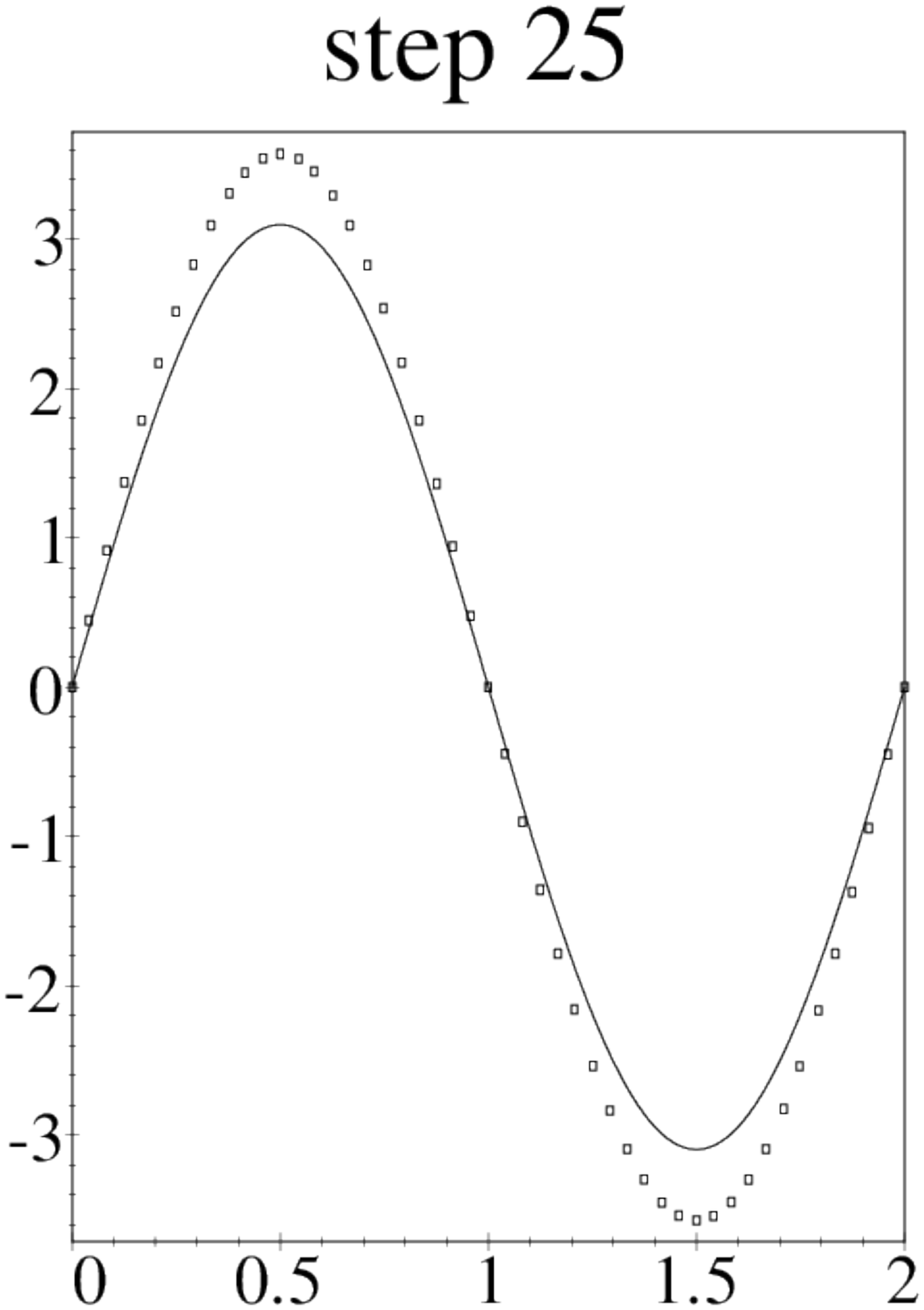}
             \epsfxsize3.3cm \epsfysize4cm \epsfbox{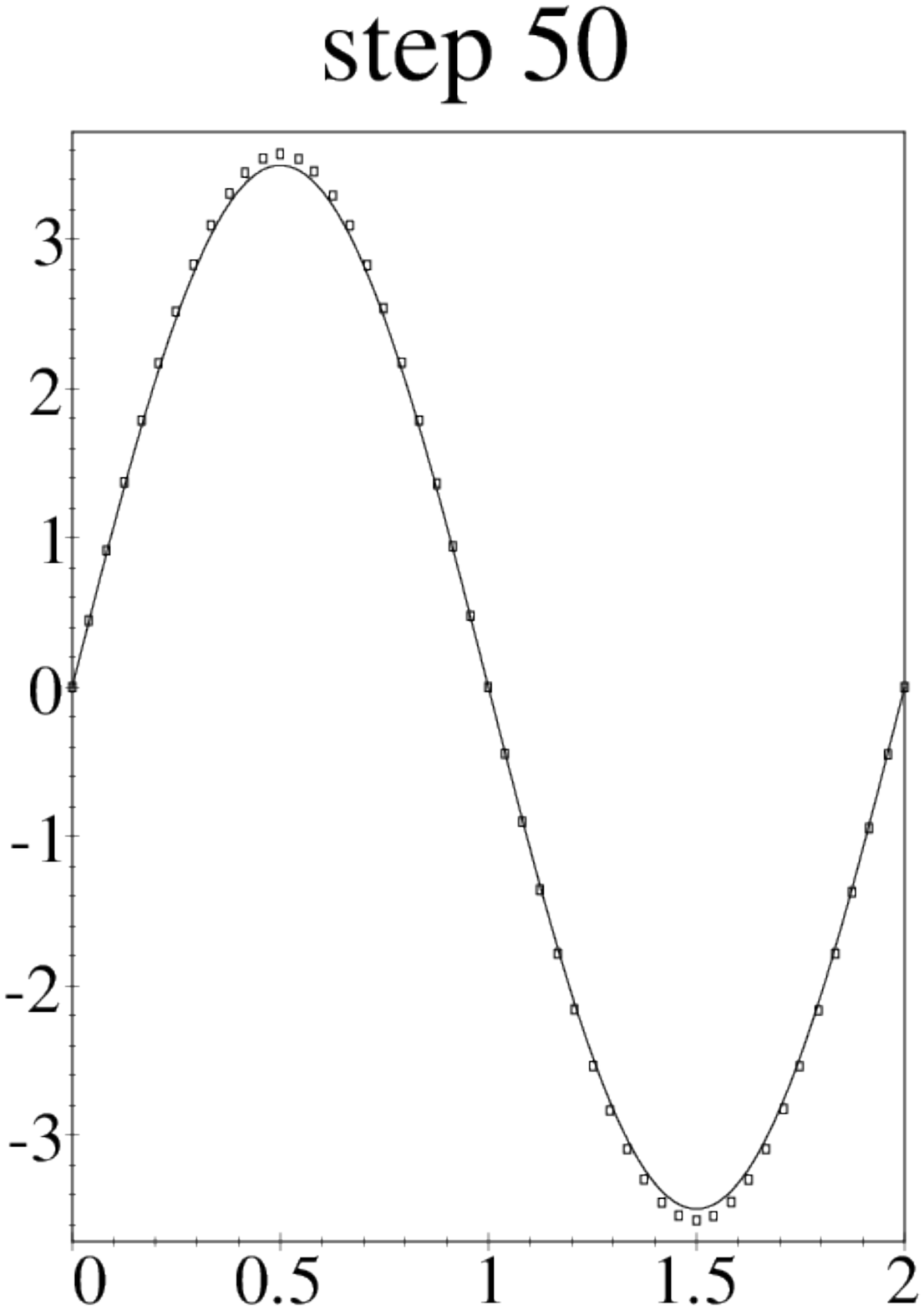}
             \epsfxsize3.3cm \epsfysize4cm \epsfbox{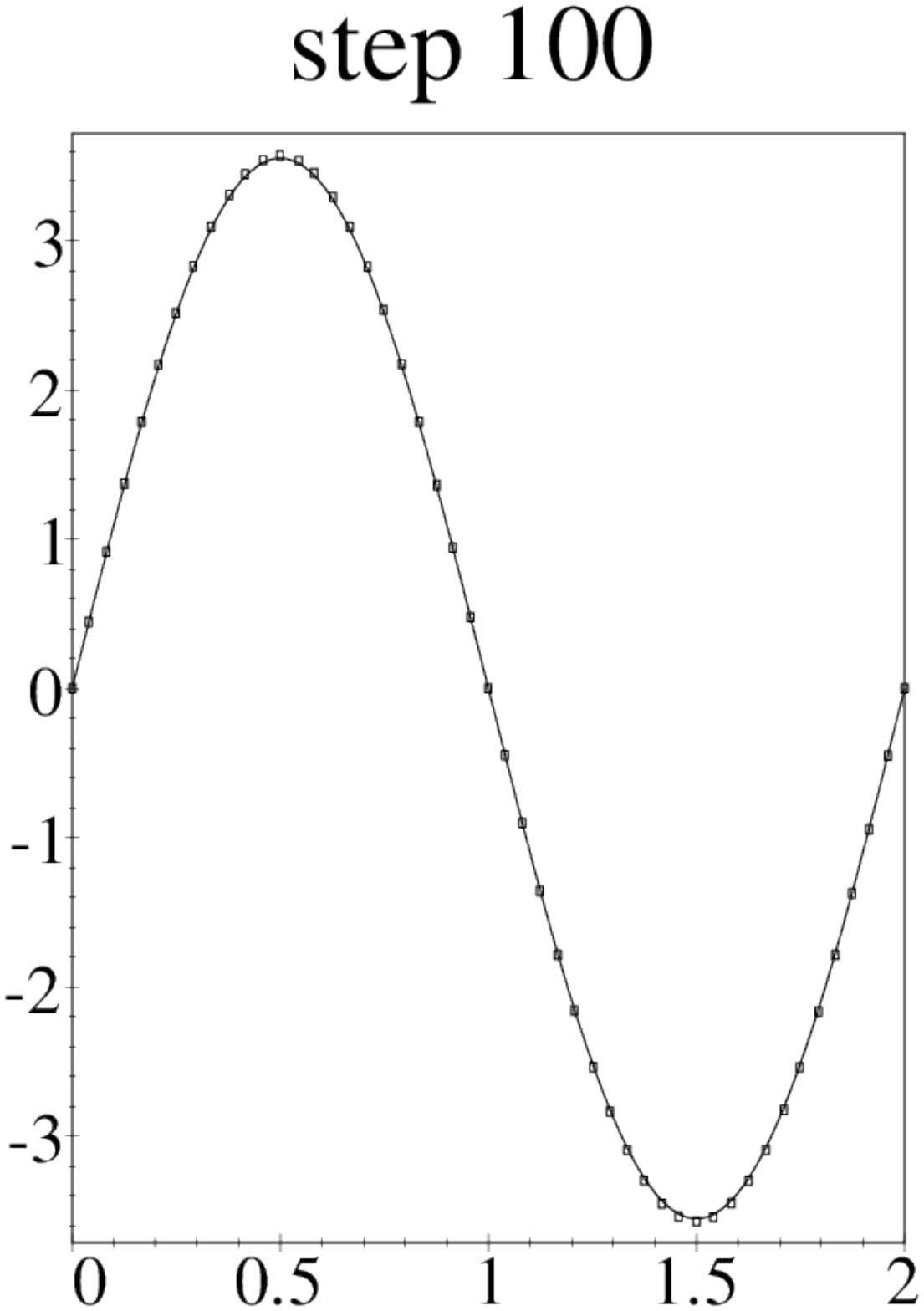} }
\caption{Iterated sequence at the unknown boundary for a linear Cauchy 
problem at a ring domain with inner and outer radius respectively $1$ and $7$ 
\label{fig3}}
\end{figure}
%------------------

        The iteration gives for this Cauchy problem a better approximation 
than it does in section~3.1. The reason for this behavior is that the 
eigenvalues of the operator $T_l$ (defined in (\ref{Tl_def})) converge to 
one slower, i.e., they are smaller than we estimated in section~3.1 (see 
Figure~\ref{fig1}).

%============================================================================%
\subsection{A linear inconsistent problem}

        In this example we take $\Omega = (0,1)\times(0,1/2)$ and define the 
boundary segments:
$$         \Gamma_1 := \{ (x,0) ; x \in (0,1) \}\, , \ \ \ \ \ \ 
           \Gamma_2 := \{ (x, 1/2) ; x \in (0,1) \   $$
$$        \Gamma_3 := \{ (0,y) ; y \in (0,1/2) \}\, , \ \ \ \,
          \Gamma_4 := \{ (1,y) ; y \in (0,1/2) \} .  $$
\noindent  For $n \in \N$ we define at $\Gamma_1$ the functions:
$$        f(x) =    0 \,  \ \ \ \ \ \ \ \ {\rm and} \ \ \ \ \ \ \ \ 
          g(x) = \left\{ \begin{array}{ll} 
                 n - n^2 |x-\pi/2| & ,\ |x-\pi/2| \leq 1/n \\
                 0                 & ,\ {\rm otherwise}
                 \end{array} \right. .
$$
\noindent  Using the reflection principle of Schwartz (see [GiTr] or [Le]) 
one proves that the Cauchy problem
$$  \left\{ \begin{array}{ccl}
            \Delta \, u  & = & 0 \ ,\ \mbox{ in } \Omega \\
            u            & = & f \ ,\ \mbox{ at } \Gamma_1 \\
            u_{\nu}      & = & g \ ,\ \mbox{ at } \Gamma_1
            \end{array} \right. ,  $$
\noindent  has an analytical solution only if $g$ is itself analytical. So 
with our choice of data we know a priori that the respective Cauchy problem 
has no classical solution.

        We take $n=100$ in the definition of $g$ and add to the Cauchy 
problem above the over determinating condition: $u = 0$ at $\Gamma_3 \cup 
\Gamma_4$. The iteration is performed as before over a 262\,913 node mesh 
using linear elements and the same stopping criterion as in section~3.1. 
In Figure~\ref{fig4} the difference $\vphi_k - \vphi_{k-1}$ is plotted for 
some values of $k$.

%-------- Figure 4
\begin{figure}
\centerline{ \epsfxsize5cm\epsfysize5cm \epsfbox{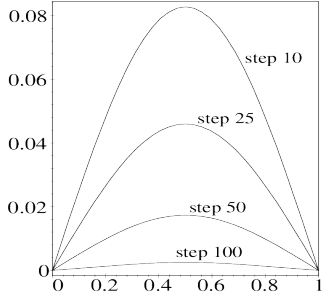} }
\caption{Difference $(\vphi_k - \vphi_{k-1})$ calculated for a linear 
inconsistent Cauchy problem \label{fig4}}
\end{figure}
%------------------
%
        The sequence $\vphi_k$ converges in the $||\cdot||_\infty$ as fast as 
it does in section~3.1, but this does not mean that it converges to a 
solution of the Cauchy problem (see Remark~\ref{contraction}).

       Both sequences $w_k$ and $v_k$ of $H^1$--functions generated in (IT) 
converge to the solution of the Cauchy problem, when this problem does have a 
solution (see Theorem~\ref{Tk_konverg}). In this example, when we analyse the 
sequences $w_k$ and $v_k$, we note that they are not converging in 
$H^1(\Omega)$ to the same limit. In Figures~5~(a) and 5~(b) 
we show the functions $w_{100}$ and $v_{100}$ respectively.

        One observes that on $\Gamma_2$ we have\, $w_{100} \simeq v_{100}$ 
and\, $(w_{100})_\nu \simeq (v_{100})_\nu$. The difference $w_k - v_k$ 
generates also a sequence of $H^1$--functions with vanishing Cauchy data at 
$\Gamma_2$ but that does not converge to zero at $\Omega$. Such examples are 
known to exist due to Hadamard (see section~1.2).
\bigskip \bigskip

%-------- Figure 5
\noindent { \epsfxsize6.0cm \epsfbox{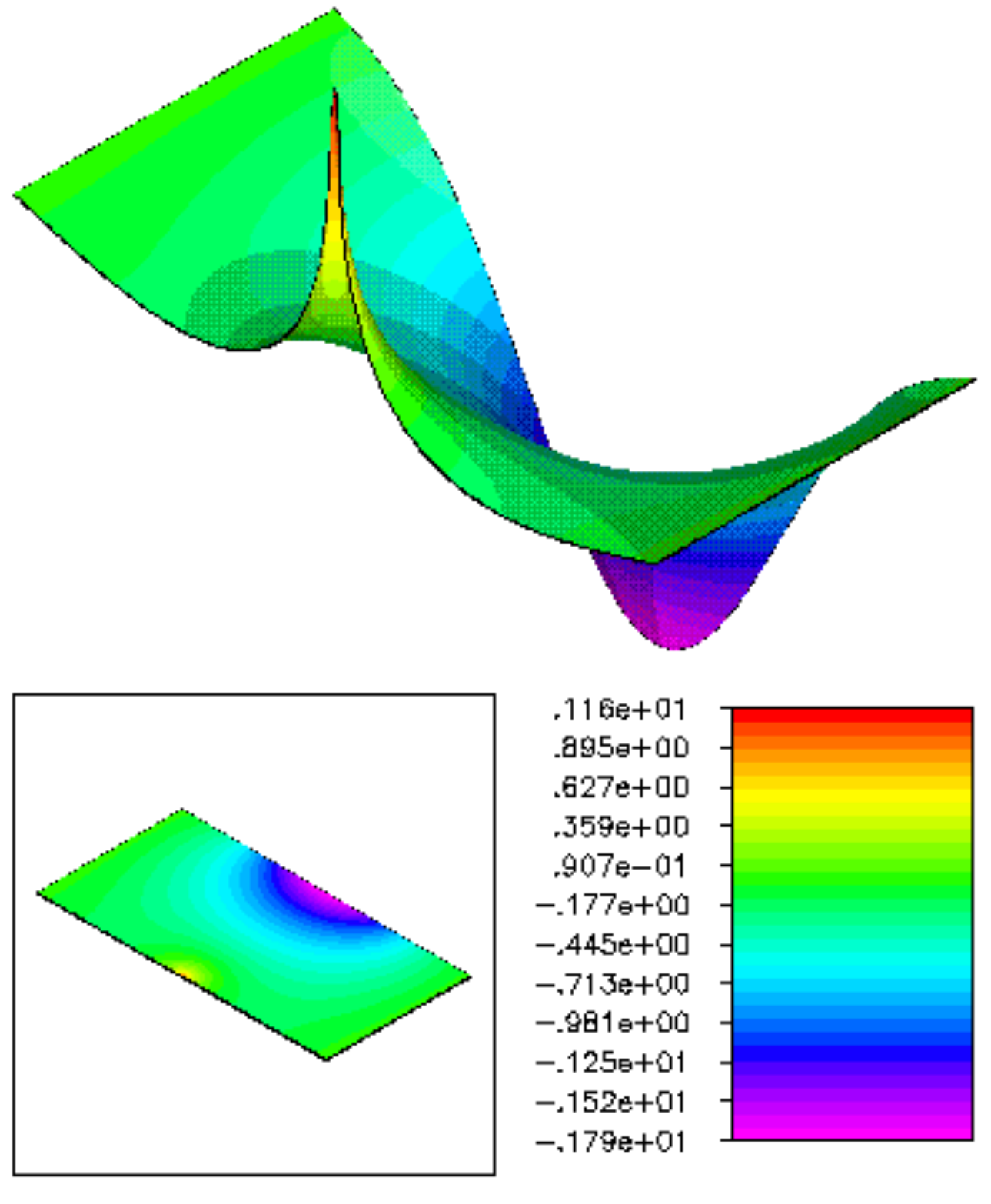} \hfill
            \epsfxsize6.0cm \epsfbox{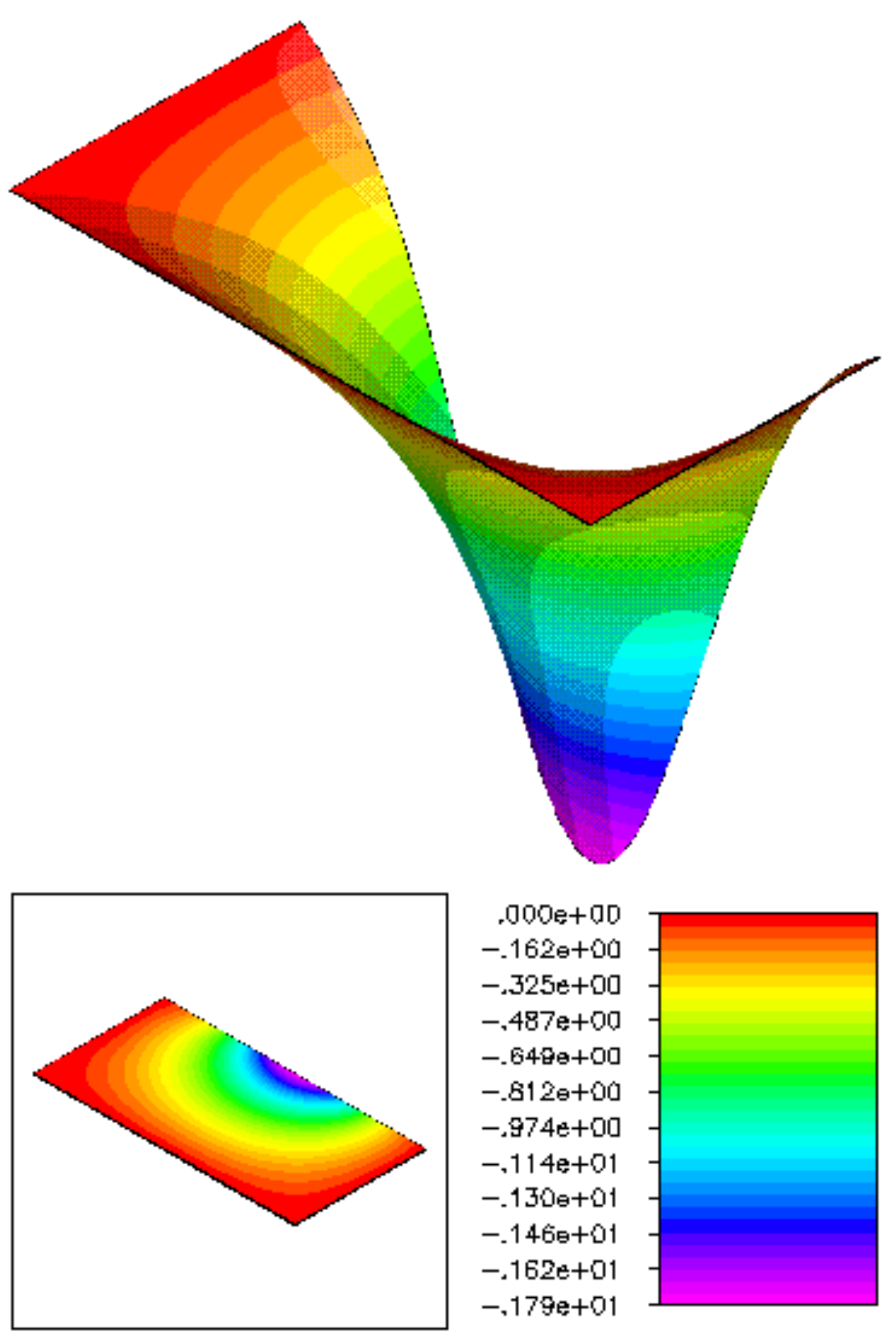} }

\centerline{\rm (a) \hskip6.5cm (b)}
\medskip

\noindent \begin{minipage}{14cm}
Figure~5: Functions $w_k$ and $v_k$ -- at (a) and (b) respectively -- after 
$100$ steps for a linear inconsistent Cauchy problem
\end{minipage}
%------------------

%============================================================================%
\subsection{A non linear problem}

        Take $\Omega$ the annulus centered at the origin with inner and outer 
radius respectively 1/2 and 1. This time we decompose the outer component of 
the boundary in two different ways:  $\partial\Omega = \Gamma_1 \cup \Gamma_2
= \tilde{\Gamma_1} \cup \tilde{\Gamma_2}$ where
$$        \Gamma_1 := \{ (x,y)\, |\ x^2 + y^2 = 1, \ x<0 \} ,
\ \ \
          \Gamma_2 := \{ (x,y)\, |\ x^2 + y^2 = 1, \ x>0 \}  $$
\noindent and
$$ \tilde{\Gamma_1} := \{(x,y)\, |\ x^2 + y^2 = 1, \ x<\sqrt{2}/2\} ,
\ \ \ 
\tilde{\Gamma_2} := \{(x,y)\, |\ x^2 + y^2 = 1, \ x>\sqrt{2}/2 \} . $$
\noindent  The inner component of $\partial\Omega$ is called $\Gamma_i := \{ 
(x,y) ;\ x^2 + y^2 = 1/4 \}$.

\noindent
Given the Cauchy data\, $f(\theta) = \sin(\theta)$\, and\, $g(\theta) = 0$\, 
on $\Gamma_1$ (respectively $\tilde{\Gamma_1}$) we reconstruct on $\Gamma_2$ 
(respectively $\tilde{\Gamma_2}$) the trace of the solution of the non linear 
Cauchy problems:
$$  \left\{ \begin{array}{rcl}
            \Delta\, u + u^3 & = & [\, (r+1/r) \sin(\theta) /2 \, ]^3
                                    \ ,\ \mbox{ in } \Omega \\
            u                & = & f \ ,\ \mbox{ at } \Gamma \\
            u_{\nu}          & = & g \ ,\ \mbox{ at } \Gamma \\
            u                & = & \frac{5}{4} \sin(\theta)
                                    \ ,\ \mbox{ at } \Gamma_i
            \end{array} \right. ,  $$
\noindent  where $\Gamma$ stands for both $\Gamma_1$ and $\tilde{\Gamma_1}$. 
Both problems have the same solution $u^\ast(x,y) = (r+1/r) \sin(\theta)/2$.

        A mesh with 82\,688 nodes and linear elements is used, the stopping 
criterion being the same as in section~3.1. In Figure~6~(a) and 6~(b) one 
can see the exact solution (dotted line) and the iterated sequence $\psi_k$ 
(solid line) for the Cauchy problems with data given on $\Gamma_1$ and 
$\tilde{ \Gamma_1}$ respectively. (note that the x--axis is parameterized 
from $0$ to $\pi$ in Figure~6~(a) and from $\pi/4$ to $3\pi/4$ in Figure~6~(b))

%-------- Figure 6
\centerline{ \hfill 
       \epsfxsize4cm \epsfysize5cm \epsfbox{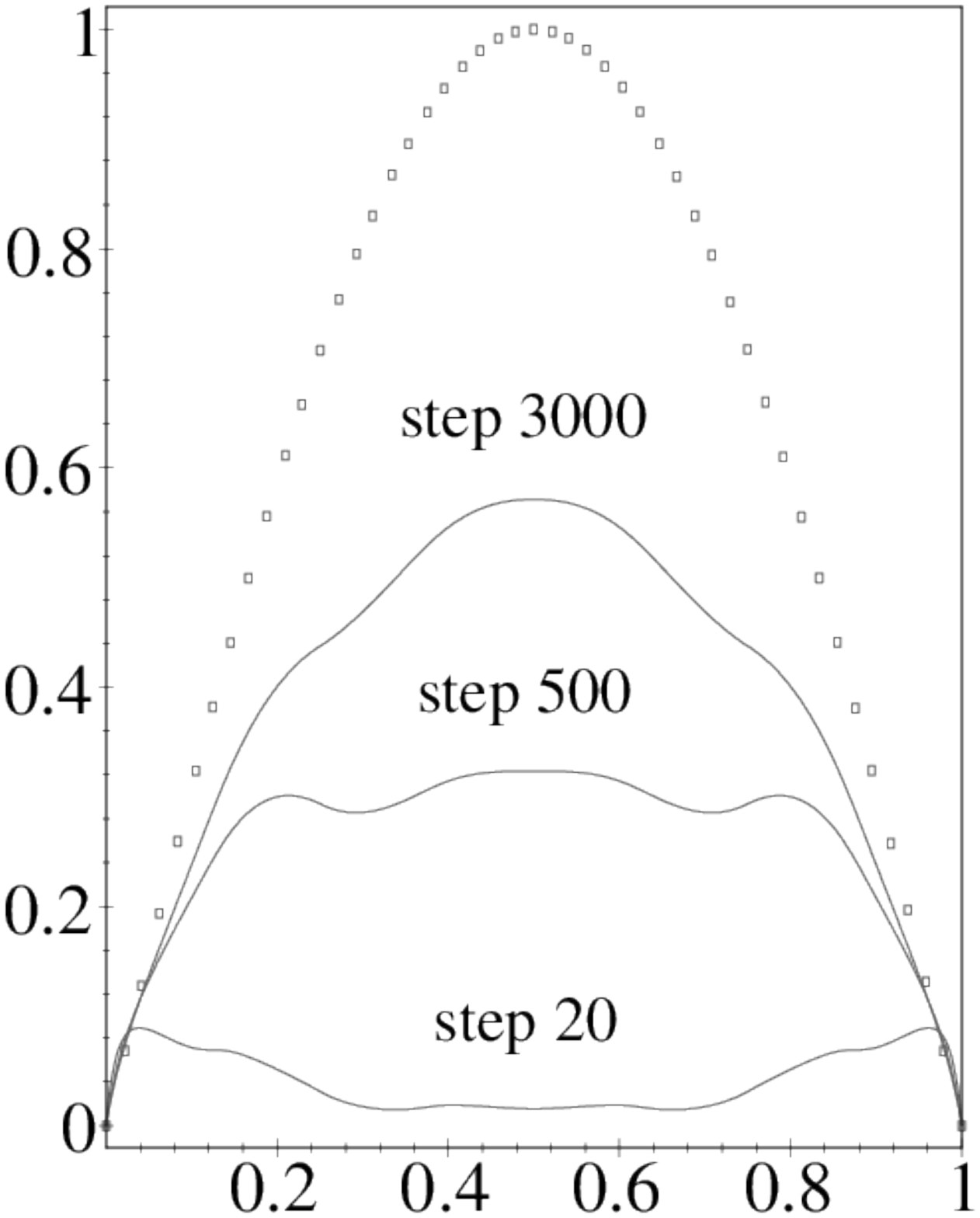} \hfill
       \epsfxsize4cm \epsfysize5cm \epsfbox{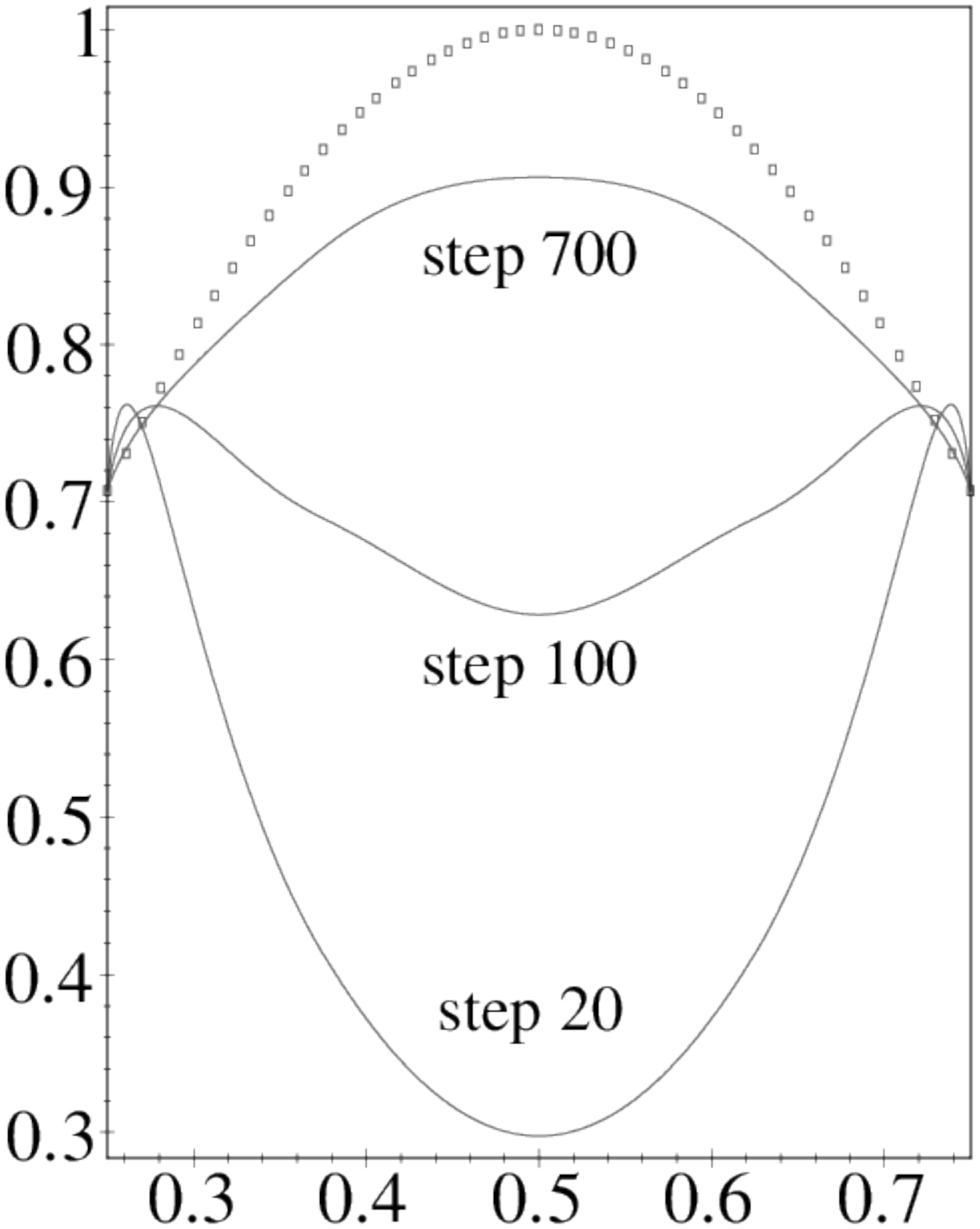} \hfill }
\vskip-.3cm
\centerline{\rm  (a) \hskip5.3cm (b)}
\bigskip

\noindent \begin{minipage}{14cm}
Figure~6: Trace of the iteration for a non linear Cauchy problem at the 
annulus with inner and outer radius respectively $1/2$ and $1$
\end{minipage}
%------------------
\bigskip

        Both iterations begin with $\psi_0 \equiv 0$. As it happens in the 
linear case, the limit of the sequence $\psi_k$ does not depend on the choice 
of the initial data $\psi_0$, but the convergence of the iteration is slower 
for this non linear operator.

        Comparing the iteration for both Cauchy problems with data at 
$\Gamma_1$ and $\tilde{\Gamma_1}$ is clear that the amount of known 
information one has, determines both velocity of convergence and precision 
of the reconstruction.

\section{Concluding Remarks}

\begin{remark}  The classical problem of Hadamard in section~3.1 is 
numerically treated in [FaMo]. They propose a direct method to solve the 
Cauchy problem on the square $(0,1)\times(0,1)$ and obtained a reconstruction 
with approximately 5\% error in the $L^2(\Gamma_2)$ norm, while our 
reconstruction has an error of approximately 4\% in this norm (see 
Figure~\ref{fig2} after 100 iterations).

        The Cauchy problem in section~3.2 is numerically treated in [KuIs], 
where boundary elements are used on the annulus with inner and outer 
radius 2 and 6 respectively. The reconstruction error in the $L^\infty( 
\Gamma_2)$--norm has the order of $10^{-2}$. The corresponding error after 
100 iteration steps has the same order (see figure~\ref{fig3}).

        The difficulty in solving Cauchy problems by using direct methods 
is that the ill-posedness of the resulting system increases as one tries to 
refine the numerical model.
\end{remark}

\begin{remark}  A motivation for the convergence of the iteration (in the 
$L^\infty$--norm) even in the inconsistent case can be found on 
the fact that the numerical discretization $T_{l,h}$ of the operator $T_l$ 
we use to generate the sequence  $\vphi_k$ is contractive. In fact, the 
finite element method has the property that only the spaces, and not the 
differential equation itself, are discretizated (see [Od]). This 
implies that the eigenvalues of $T_{l,h}$ are lower or equal to the 
respective eigenvalues of $T_l$. As the eigenfunctions of $T_l$ (defined at 
$\Gamma_2$) cannot be obtained by taking traces of the linear finite elements 
we are using, we conclude that the eigenvalues%
\footnote{Obviously there is only a finite number of them.}
of $T_{l,h}$ are strict smaller than the respective eigenvalues of $T_l$.
\label{contraction} \end{remark}

\begin{remark}  Another discretization of the operator $T_l$ was tested. We 
tried to solve the mixed problems on the square domain using finite 
differences. This has the advantage of being faster (we must invert 2 
stiffness matrices once, but each iteration is than evaluated as a simple 
matrix vector product) specially if the number of iteration to be computed 
is large. Our numerical experiments show that the eigenvalues of this second 
discretization of $T_l$ approximate the real eigenvalues better than the 
discretization by finite elements does. As a direct consequence, the 
approximation obtained using a finite difference discretization --with the 
same stopping criterion-- is better than the one obtained using finite 
elements.
\end{remark}

\begin{remark}  If the differential operator $P$ in (\ref{L-definition}) is 
non linear, the property
$$  T \, \vphi \ = \ \vphi ,  $$
\noindent  where $T$ is the operator defined in (\ref{Tl_def}), will only 
hold if one can guarantees the unicity of solutions for the Cauchy problem 
governed by $P$ {\em a priori}. Even if this is the case, one still have to 
make sure that each mixed problem on the iteration is uniquely solvable.

        We does not prove these properties for the Cauchy problem in 
section~3.4, but the iteration was performed using different analytical 
functions as initial data. The velocity of convergence and the limit 
encountered is the same for all our tests.

        In spite of the fact that the convergence of the iteration is much 
slower in the non linear case, it can be again verified that the fixed point 
of $T_h \, \vphi = \vphi$ approaches better the fixed point of 
$T \, \overline{\vphi} = \overline{\vphi}$, if the precision $h$ of the 
discretization increases.
\end{remark}

\appendix

%----------------------------------------------------------------------------%
%                                    A                                       %
%----------------------------------------------------------------------------%
\section{Sobolev spaces and trace Theorems}

        Let $\Omega \in \R^2$ be an open, bounded, regular%
\footnote{We mean $\Omega$ is locally at one side of $\partial\Omega$.} 
set with $C^\infty$--boundary $\partial\Omega$, which is splited in 
$\partial\Omega = \cup_{j=1}^N \overline{\Gamma_j}$, the subsets 
$\Gamma_j$ being open, connected and satisfying $\Gamma_i \cap \Gamma_j = 
\emptyset$ for  $1 \leq i \not= j \leq N$. We denote by $\gamma_j$ the 
trace operator with domain $C^\infty(\overline{\Omega})$ and 
range $C^\infty(\Gamma_j)$. With $\nu_j$ we represent the vector field 
normal to $\Gamma_j$.

        Given the second order elliptic operator
$$  P(u) \ := \ - \sum_{i,j=1}^{2}{\, D_i (a_{i,j} D_j u)} \ + \ 
                          \sum_{i=1}^{2} {\, a_i D_i u} \ + \ a_0\, u ,  $$
\noindent  we represent the co-normal derivative of a function $u$ 
respective to $P$ by
$$  u_\nuA \ = \ \frac{\partial u}{\partial \nuA} \ := \ 
                       \sum_{i,j=1}^{2}{\, a_{i,j} \, \nu_i \, (D_j u)} \,  $$
\noindent  where $A$ represent the $2\times 2$ matrix $(a_{i,j})$. 

        We introduce now the Sobolev spaces used in this article. For $s = k 
+ \sigma \in \R^+$ with $k \in \N_0$ and $\sigma \in [0,1)$, we define
$$  H^s(\Omega) := \overline{C^\infty(\overline{\Omega})}^
                                          {||\cdot||_{s;\Omega}} ,\
    H^s_0(\Omega) := \overline{C^\infty_0(\Omega)}^{||\cdot||_{s;\Omega}} ,\
    H^s_0(\Omega\cup\Gamma) := \overline{C^\infty_0(\Omega\cup\Gamma)}^
                                                   {||\cdot||_{s;\Omega}},  $$
\noindent  where the functional $||\cdot||_{s;\Omega}$ is defined by
$$  ||u||_{s;\Omega}^2 \ := \ 
    \sum\limits_{|\alpha|\leq k}{||D^\alpha u||_{L^2(\Omega)}^2} \ + \ 
    \sum_{|\alpha|=k} \ {\int\!\!\!\int}_{\!\!\Omega\times\Omega}
{\frac{|D^\alpha u(x) - D^\alpha u(y)|^2}{|x - y|^{n+2\sigma}}\ dx\, dy}\, .$$
\noindent  For $s>0$ we define the space $H^{-s}(\Omega)$ by duality
$$    H^{-s}(\Omega) \ = \ \{ u\in {\cal D}'(\Omega)\ /\ <u,\cdot>_
                                 {L^2(\Omega)}\ \in H^s_0(\Omega)' \} \, .  $$
\noindent  Given the differential operator $P$ as above, we define the space 
$H^1(\Omega;P)$ as the space of distributions with
$$           \{ u\in H^1(\Omega)\ /\ P \, u \in L^2(\Omega) \} \, .         $$
\noindent  The last space we need is $H^s_{00}(\Omega), \ s\in \R^+$. If 
$s = k + \sigma$ with $k \in \N_0$ and $\sigma \in [0,1)$ we define the 
functional
$$  ||u||_{s;{00};\Omega} \ := \ \left\{ ||u||^2_{s;\Omega} \ + \
          \sum_{|\alpha|=k} \, \int_{\Omega} {\, |D^\alpha u(x)|^2 \,
                      d^{-2\sigma}(x,\partial\Omega) \ dx} \right\}^{\me},  $$
\noindent  where $d(x,\partial\Omega)$ is the Euclidean distance between 
$x \in \R^n$ and $\partial\Omega$. Now we define $H^s_{00}(\Omega)$ as the 
closure of $C_0^\infty(\Omega)$ with respect to $||\cdot||_{s;{00};\Omega}$.

        The next theorems can be found in [DaLi] or [Gr1,2] and will be 
needed in the study of weak solutions of mixed boundary value problems.

%=====================================%
\begin{theorem}[Neumann trace of a $H^1(\Omega;P)$ distribution]  The operator
$$  \Gamma_\nuA : u \ \longmapsto \ \gamma \frac{\partial u}{\partial \nuA} $$
\noindent  defined on $C^\infty (\overline{\Omega})$ has only one continuous 
extension defined at
$$  \tilde{\Gamma}_\nu : H^1(\Omega;P) \ \longrightarrow \
                                            H^{\mme}(\partial \Omega) \, .  $$
\label{spur_H_om_P} \end{theorem}
%
%=====================================%
\begin{theorem}[Dirichlet trace on $\gamma_j$]  Let $m \in \N$. The operator
$$  u \ \longmapsto 
 \ \left\{ \gamma_j u,\ \gamma_j \frac{\partial u}{\partial \nu_j},
 \ \dots\ ,\ \gamma_j \frac{\partial^{m-1}u}{\partial \nu^{m-1}_j} \right\} $$
\noindent  defined for every $j=1,\dots,m$ at $C^\infty(\overline{\Omega})$ 
has only one continuous extension from
$$    H^m(\Omega)\ \ \mbox{onto}\ \ \prod_{i=0}^{m-1}{H^{m-i-\frac{1}{2}}
                                                         (\Gamma_j)} \, .   $$
\label{spursatz_rand_st} \end{theorem}
\begin{remark}  The trace operator in Theorem~\ref{spursatz_rand_st} has a 
continuous right inverse. This fact follows from the existence of a 
prolongation operator $\omega \in {\cal L}(H^s(\Gamma_j), 
H^s(\partial\Omega))$, $s \geq 0$, that is the continuous right inverse of 
the restriction operator $\rho \in {\cal L} (H^s(\partial\Omega),$ 
$H^s(\Gamma_j))$. For details see [Au] pp. 187-194.
\label{inv_ spur_dir_rand_st} \end{remark}
%
%=====================================%
\begin{theorem}[Neumann trace on $\gamma_j$]  The Neumann trace operator 
defined on $C^\infty(\overline{\Omega})$ with range in $C^\infty(\Gamma_j)$ 
has only one continuous extension as an operator from
$$       H^1(\Omega;P)\ \ \mbox{into}\ \ H^{\me}_{00}(\Gamma_j)',      $$
\noindent   for $j = 1,\dots,N.$
\label{spur_H_om_P_rand_st} \end{theorem}

%----------------------------------------------------------------------------%
%                                    B                                       %
%----------------------------------------------------------------------------%
\section{Green's formula}

        In the analysis of the mixed boundary value problems that appear in 
the iterative procedure (IT), we make use of a special version of the 
(first) Green formula. The results presented here are still valid if 
$\partial\Omega$ only Lipschitz continuous and they can be found in [Au], 
[DaLi], [Gr1], [Le], [LiMa] or [Tr].

%=====================================%
\begin{theorem}  If $(u,v)$ is a pair of functions in $H^2(\Omega) \times 
H^1(\Omega)$, then we have
\begin{equation}   \int_{\Omega}{P u\, v\ dx }\ =\ 
                                  \int_{\partial\Omega}{u_\nuA\, v\ d\Gamma}
         \ - \ \int_{\Omega}{(\nabla v)^t A \, \nabla u\ dx} \, .
\label{green_standard} \end{equation}
\noindent  Equation (\ref{green_standard}) is still valid if 
$(u,v) \in H^1(\Omega;P) \times H^1(\Omega)$.%
\footnote{The last assertion follows from the density of the embedding 
of $C^\infty(\overline{\Omega})$ into $H^1(\Omega;\Delta)$ together with 
Theorem~\ref{spur_H_om_P}.}
\label{sa:green_standard} \end{theorem}

        In next theorem the boundary integral in (\ref{green_standard}) 
is replaced by a sum of integrals over each $\gamma_j$.

%=====================================%
\begin{theorem}  Given functions $u$ in $H^2(\Omega)$ and $v$ in $H^1(\Omega)$ 
we have
\begin{equation}   \int_{\Omega}{P u\, v\ dx }\ = \ 
                   \sum_{j=1}^{N}{\ \int_{\Gamma_j}{u_\nuA\, v\ d\Gamma}}
                   \ - \ \int_{\Omega}{(\nabla v)^t A \, \nabla u\ dx}.
\label{green_rand_stuck} \end{equation}
\label{sa:green_rand_stuck} \end{theorem}

        This theorem is a consequence of the trace Theorem~\ref
{spursatz_rand_st}, that guarantees the boundedness of the inner products 
in (\ref{green_rand_stuck}) in the sense of $L^2(\Gamma_j)$. 
Theorem~\ref{sa:green_rand_stuck} still holds if $\partial\Omega$ is a 
$C^{1,1}$ polygon. In the next theorem (see [Le]) we formulate Green's 
formula in the exact context needed in this paper.

%=====================================%
\begin{theorem}  Given $u \in H^1(\Omega;P)$ and $ v \in \{ v\in H^1 
(\Omega)\, |\, \gamma_j\, v \in H^{\me}_{00}(\Gamma_j)$, $j=1,\dots,N \}$ 
we have
$$    \int_{\Omega}{P u \, v\ dx}\ = \ 
      \sum_{j=1}^{N}{\ \int_{\Gamma_j}{u_\nuA\, v\ d\Gamma}}
                  \ - \ \int_{\Omega}{(\nabla v)^t A \, \nabla u\ dx} .  $$
\label{green_allgemein} \end{theorem}

        The next two theorems describe some characteristics of the trace 
of $H^1$--functions, that are needed in the formulation of the iterative 
procedure (IT).

%=====================================%
\begin{theorem} Let $N = 2$, i.e. $\partial\Omega = \overline{\Gamma_1} \cup 
\overline{\Gamma_2}$ and $u \in H^1(\Omega)$. If $\gamma_1 \, u$ is a 
$H^{\me}_{00}(\Gamma_1)$ distribution, then $\gamma_2 \, u$ is a 
$H^{\me}_{00}(\Gamma_2)$ distribution.
\label{H_me00_eigenschaft} \end{theorem}

%=====================================%
\begin{theorem}  Let $N = 2$, i.e. $\partial\Omega = \overline{\Gamma_1} \cup 
\overline{\Gamma_2}$ and $u$ be a P--harmonic function in $H^1(\Omega)$. 
If $\gamma_1 \frac{\partial u}{\partial\nuA}$ belongs to 
$H^{\mme}(\Gamma_1)$, than $\gamma_2 \frac{\partial u}{\partial\nuA}$ 
belongs to $H^{\mme}(\Gamma_2)$.
\footnote{One should note that $H^{\mme}(\Gamma_2) = H^{\me}(\Gamma_2)'$, 
because of the identity $H^{\me}_0(\Gamma_2) = H^{\me}(\Gamma_2)$.}
\label{H-me_eigenschaft} \end{theorem}

%----------------------------------------------------------------------------%
%                                    C                                       %
%----------------------------------------------------------------------------%
\section{Mixed boundary value problems}

        For the analysis of mixed problems we need a type of {\em 
Poincar\'e inequality} on the space $H^1_0(\Omega\cup\Gamma_j)$. This 
is obtained with theorem

%=====================================%
\begin{theorem}  Given a function $u \in H^1_0(\Omega\cup\Gamma_j)$ we have
$$  ||u||_{L^2(\Omega)} \ \leq \ c \, ||\nabla u||_{L^2(\Omega)} ,  $$
\noindent   where the constant $c$ depends only on $\Omega$.%
\footnote{For a detailed proof see [Tr] pp. 69.}
\label{poincarre} \end{theorem}

        We analyze the following mixed problem problem at $\Omega$. Given 
the functions $f \in H^{\me}(\Gamma_1)$ and $g \in H^{\me}_{00}(\Gamma_2)'$, 
find a $H^1$--solution of

\medskip \noindent
$  (GP)     \hfill \left\{ \begin{array}{rl}
                   \Delta u = 0  & ,\ \mbox{in}\ \Omega \\
                   u = f         & ,\ \mbox{at}\ \Gamma_1 \\
                   u_\nu = g     & ,\ \mbox{at}\ \Gamma_2
            \end{array} \right. . \hfill $
\medskip

\noindent  Existence, unicity and continuous dependency of the data for (GP) 
are given by the following theorem of Lax--Milgramm type.

%=====================================%
\begin{theorem}  For every pair of data $(f,g) \in H^{\me}(\Gamma_1) \times 
H^{\me}_{00}(\Gamma_2)'$ the problem (GP) has a unique solution $u \in 
H^1(\Omega)$. Further it holds
\begin{equation}
    ||u||_{H^1(\Omega)} \ \leq \ c \left( ||f||_{H^{\me}(\Gamma_1)}\ +\
                                      ||g||_{H^{\me}_{00}(\Gamma_2)'} \right).
\label{ungl_stetig_abhang} \end{equation}
\label{ex_eind_gemischt} \end{theorem}

        The next theorem investigates the regularity of the $H^1$--solution 
of (GP). For a detailed proof see [Gr1,2] or [Wn].

%=====================================%
\begin{theorem}  For boundary data $f \in H^{\tme}(\Gamma_1)$ and $g \in 
H^{\me}(\Gamma_2)$, the $H^1$--solution $u$ of $(GP)$ can be written as
\begin{equation} u \ =\ h \ +\ \sum_{i=1}^2{\alpha_i\,u_i}\, , \end{equation}
\label{darstel_losung_gemischt} \end{theorem}

\noindent  where $h \in H^2(\Omega)$, $\alpha_i \in \R$ and $u_i$ are the 
singular $H^1$--functions
$$  u_i(r,\theta) \ = \ r_i^{\me} \sin{\frac{\theta_i}{2}} .  $$

\noindent  Here $r_1$ (respectively $r_2$) is the distance from $z = 
(r,\theta) \in \Omega$ to the contact point $p_a$ between $\Gamma_1$ and 
$\Gamma_2$ (respectively $p_b$ between $\Gamma_2$ and $\Gamma_1$); $\theta_1$ 
(respectively $\theta_2$) is the angle between $z-p_a$ (respectively 
$z-p_b$) and the line tangent to $\partial\Omega$ at $p_a$ (respectively 
at $p_b$) in the direction of $\Gamma_1$.

%----------------------------------------------------------------------------%
%                                    D                                       %
%----------------------------------------------------------------------------%
\section{Unicity results for Cauchy problems}

        What we present now is a generalisation of some classical results 
concerning the theory of diferential operators with analytical coefficients. 
We use the Cauchy--Kowalewsky and Holmgren theorems%
\footnote{The Cauchy--Kowalewsky theorem can be found in [Fo] pp. 69, [DaLi] 
or [Jo]; for details on the Holmgren theorem one may see [Jo] pp. 65 or 
[DaLi].}
together with a regularity theorem for weak solutions of elliptic 
equations, to guarantee the uniqueness of $H^1$--solutions of Cauchy 
problems.%
\footnote{See also [Is].}

%=====================================%        
\begin{theorem}  Let $L$ be a linear differential operator of order 2 with 
$C^\infty(\Omega)$--coefficients, where $\Omega \subset \R^n$ is an open 
regular set. Define $a(\cdot,\cdot)$ the respective bilinear form, which is 
supposed to be strong coercive. If the distribution $u$ is a solution of 
$Lu = \psi$ with $\psi \in H^k_{\rm loc}(\Omega),\ k\in \N$, then 
$u \in H^{k+2}_{\rm loc}(\Omega)$.%
\footnote{The theorem still holds for differential operators of order $2m$. 
In this case we conclude $u \in H^{k+2m}_{\rm loc}(\Omega)$.}
\label{lokal_Reg}\end{theorem}

%=====================================%
\begin{remark}  Under the same assumptions as in Theorem~\ref{lokal_Reg} it 
follows from the assumption $\psi \in C^\infty(\Omega)$ that $u$ belongs to 
$C^\infty(\Omega)$.
\label{RegKor} \end{remark}

        We can now state the uniqueness result for Cauchy problems.

%=====================================%
\begin{theorem}  Let $\Omega$ be an open, bounded and simply connected set 
of $\R^2$ with analytical boundary $\partial\Omega$. Let $\Gamma$ be an 
open simply connected subset of $\partial\Omega$ and the differential 
operator $L$ defined as in Theorem~\ref{lokal_Reg}. Then the Cauchy--Problem
$$                   \left\{  \begin{array}{ll}
                          Lu = \psi       &,\, {\rm in}\ \Omega \\
                          u = f           &,\, {\rm at}\ \Gamma \\
                          u_\nu = g       &,\, {\rm at}\ \Gamma
                     \end{array} \right.  $$
\noindent  has for $\psi \in L^2(\Omega)$, $f \in H^{\me}(\Gamma)$ and $g 
\in H^{\me}_{00}(\Gamma)'$ at most one solution in $H^1(\Omega)$.
\label{schwach_eindeut} \end{theorem}

        For the special case $L$ = $\Delta$, the Laplace operator, the 
assumptions relative to $\Omega$ can be weakened. For this operator 
Theorem~\ref{schwach_eindeut} still holds even if $\Omega$ is not supposed
to be simply connected. \bigskip \bigskip

\noindent {\Large\bf Acknowledgments}

\noindent
        The author wants to thank his PhD supervisor Prof. Dr. J. Baumeister 
for the inspiring discussions that contributed for the conclusion of this 
paper. The author was supported during his stay at the Goethe 
Universit\"at in Frankfurt am Main by the Deutsche Akademische Austauchdienst.

\end{document}